\documentclass{amsart}
 \usepackage[foot]{amsaddr}
\usepackage[lite]{amsrefs}

\usepackage{amssymb}

\usepackage[utf8]{inputenc}
\usepackage[title]
  {appendix}
\usepackage{amsmath}
\usepackage{amssymb}
\usepackage{mathrsfs}
\usepackage{mathtools}
\usepackage{amscd}
\usepackage{amsthm,color}

\usepackage{stmaryrd}
\usepackage{textgreek}
\usepackage{hyperref}
\usepackage{cleveref}

\usepackage[all,cmtip]{xy}

\def\R{{\mathbb{R}}}
\def\CC{{\mathbb C}}
\def\<#1>{\langle #1\rangle}
\newcommand{\qm}[1]{\operatorname{QM}(#1)}
\newcommand{\I}[1]{\operatorname{I}\left(#1\right) }
\def\RR{{\mathbb R}}
\def\rank{{\operatorname{rank}}}
\def\fn{{\mathfrak{n}}}
\def\ol{\overline}
\def\bS{\mathbf{S}}
\def\bI{\mathbf{I}}
\def\cS{{\mathcal{S}}}
\def\fm{{\mathfrak{m}}}
\def\NN{{\mathbb N}}
\def\Res{{\mathrm{Res}}}
\def\cF{{\mathcal{F}}}
\def\cL{\mathcal{L}}
\def\cI{{\mathcal{I}}}
\def\mQ{\mathcal{Q}}
\def\cK{{\mathcal{K}}}
\def\LT{{\operatorname{LT}}}
\def\codim{{\operatorname{codim}}}
\def\fc{{\mathfrak{c}}}

\providecommand{\pacs}[1]{\textbf{\textit{MSC Classification---}} #1}

\numberwithin{equation}{section}
\theoremstyle{plain}
\newtheorem{lemma}{Lemma}
\newtheorem{theorem}{Theorem}
\newtheorem{corollary}{Corollary}
\newtheorem{ass}{Assumption}
\newtheorem{prop}{Proposition}

\theoremstyle{definition}
\newtheorem{definition}{Definition}

\theoremstyle{remark}
\newtheorem{example}{Example}
\newtheorem{remark}{Remark}

\title[Exactness and Effective Degree Bound of Lasserre's Relaxation]{Exactness and Effective Degree Bound of Lasserre's Relaxation for Polynomial Optimization over Finite Variety}

\author{Zheng Hua $^{\dagger}$}
\address{$^\dagger$Department of Mathematics\\
        The University of Hong Kong\\ Hong Kong SAR, China.}
        \address{$^\ddag$School of Mathematical Sciences\\
        Shenzhen University\\Shenzhen, China.}
\email[Zheng~Hua]{$^{1}$ zheng.hua.zju@gmail.com}
\author{Zheng Qu $^{\ddag, \ast}$}
\email[Corresponding author]{$\ddag$ quzheng2003@gmail.com}
\address{$^{\ast}$Corresponding author}

\begin{document}


	


\date{}

\begin{abstract}In this paper, we address the effective degree bound problem for Lasserre's hierarchy of moment-sum-of-squares (SOS) relaxations in polynomial optimization involving $n$ variables. We assume that the first $n$ equality constraint polynomials 
$g_1,\ldots,g_n$ do not share any nontrivial common complex zero locus at infinity and that the optimal solutions are nonsingular. Under these conditions, we derive an effective degree bound for the exactness of Lasserre's hierarchy. Importantly, the assumption of no solutions at infinity holds on a Zariski open set within the space of polynomials of fixed degrees. As a direct consequence, we provide the first explicit degree bound for gradient-type SOS relaxation under a generic condition.
\end{abstract}
\maketitle



\section{Introduction}

In a seminal work (\cite{Lasserre2001}), Lasserre  proposed  a hierarchy of convex relaxations to approximate the 
global minimum of a polynomial function over {a} compact set defined by polynomial inequalities. This hierarchy provides a sequence of monotonically converging lower bounds, which are optimal values of convex semidefinite programming (SDP) problems and thus computationally tractable.
Increasing the relaxation order  naturally leads to  tighter approximation, at the price of solving SDP programs of larger size. From the point of view of complexity analysis, it is  important to establish bounds on the order for which the relaxation  is tight, i.e., when the lower bound is equal to the global optimal value. 

\subsection{Background}\label{sec:bac}
We consider in this paper the following polynomial optimization problem:
\begin{equation}\label{eq:1}
 \begin{array}{lll}
 &\displaystyle \min_{x\in \R^n}~ & f(x)  \\
 ~~~~~~~~&\mathrm{s.t.} &   g_1(x)=\cdots=g_k(x)=0{,} \\
 && h_1(x)\geq 0,\ldots, h_m(x)\geq 0{,}
 \end{array}
\end{equation}
where  $f, g_1,\ldots,g_k, h_1,\ldots,h_m$ are  polynomial functions of $n$ variables. We consider the case {where} the complex variety $\{z\in \CC^n: g_1(z)=\cdots=g_k(z)=0\}$ is nonempty and contains finitely many points. Note that this implies $k\geq n$. Moreover we assume that the feasible region is not empty and denote by $f_*$ the optimal value of~\eqref{eq:1}. 

We recall some notations used to define the hierarchy of relaxations.  We denote by $\R[x_1,\ldots,x_n]$ the ring of polynomials in variables $x_1,\ldots,x_n$ with real coefficients. Denote by $\<g_1,\ldots,g_k>$ the ideal of $\R[x_1,\ldots,x_n]$ generated by $g_1,\ldots,g_k$:
$$
\<g_1,\ldots,g_k>:=\left\{ \sum_{i=1}^k \lambda_i g_i : \lambda_i\in \R[x_1,\ldots,x_n],\enspace \forall i=1,\ldots,k \right\}.
$$
  A polynomial $\sigma$ is a \textit{sum of squares} (SOS) if it is written as a finite sum of squares of polynomials, i.e., $\sigma=q_1^2+\cdots+q_k^2$ where $q_1,\ldots,q_k$ are polynomials.  
The set of all SOS polynomials in variables $x_1,\ldots,x_n$ is denoted
by $\Sigma [x_1,\ldots,x_n]$. Let $h_0$ denote the constant polynomial function: $h_0(x)\equiv 1$. 
The \textit{quadratic module} generated by $h_0,h_1,\ldots,h_m$ is defined as:
$$
\qm{h_0,\ldots,h_m}:=\left\{
\sum_{i=0}^m \sigma_ih_i: \sigma_i \in \Sigma[x_1,\ldots,x_n],\enspace\forall i=0,\ldots,m
\right\}{.}
$$
Let $\deg(p)$ denote the degree of a polynomial function $p$.
For a positive integer $d$,
the $d$th truncation of the ideal $\<g_1,\ldots,g_k>$ is defined as:
$$
\I{g_1,\ldots,g_k}_{d}:=\left\{ \sum_{i=1}^k \lambda_i g_i : \lambda_i\in \R[x_1,\ldots,x_n],  \deg(\lambda_i g_i) \leq d,\enspace \forall i=1,\ldots,k \right\},
$$
and the $d$th truncation of the quadratic module $\qm{h_0,\ldots,h_m}$ is defined as:
$$
\qm{h_0,\ldots,h_m}_{d}:=\left\{
\sum_{i=0}^m \sigma_i h_i: \sigma_i \in \Sigma[x_1,\ldots,x_n], \deg(\sigma_i h_i)\leq d,\enspace \forall i=0,\ldots,m
\right\}.
$$
{Denote \begin{equation}\label{eq:Qd}
\mQ_d:=\I{g_1,\ldots,g_k}_{d}+ \qm{h_0,\ldots,h_m}_{d} .
\end{equation}
Note that $\mQ_d$ is a convex cone in the real vector space $\R[x_1,\ldots,x_n]_d:=\{p\in \R[x_1,\ldots,x_n]: \deg(p)\leq d\}$. Denote by $(\R[x_1,\ldots,x_n]_d)^*$  the dual space of $\R[x_1,\ldots,x_n]_d$ and let $\cL_d$ be the dual cone of $\mQ_d$:
\begin{equation}\label{eq:cL}
\cL_d:=\{\Lambda \in (\R[x_1,\ldots,x_n]_d)^*: \langle \Lambda,q \rangle \geq 0,\enspace \forall q\in \mQ_d\}\, .
\end{equation}
}
{
Let
$$
d_0:=\left\lceil\max\left(\deg(f),\enspace\max_{i=1,\ldots,k}  \deg(g_i),\enspace\max_{i=1,\ldots,m}\deg(h_i)\right)/2\right\rceil.
$$
}
For a positive
integer {$d\geq d_0$}, 
Lasserre's  sum of squares relaxation of order $d$ of~\eqref{eq:1}  amounts to finding the value ${f^{sos}_d}$ of the following problem (\cite{Lasserre2001}):
\begin{equation}\label{eq:13}
\left\{ \begin{array}{lll}
 \displaystyle {f^{sos}_{d}}:=&\displaystyle\sup_{c\in \R} & c  \\
 ~~~~~~~~&\mathrm{s.t.} &  f-c \in {\mQ_{2d}\, .} 
 \end{array}\right.
\end{equation}
{
The standard Lagrange dual problem of~\eqref{eq:13} can be written as follows:
\begin{equation}\label{eq:d-mom}
\left\{ \begin{array}{rl}
f^{mom}_{d}:=\inf & \langle \Lambda,f \rangle \\
\mathrm{s.t.} &   \Lambda\in \cL_{2d},\enspace \<\Lambda, h_0> = 1\, ,
\end{array} \right.
\end{equation}
known as the moment relaxation of order $d$ of~\eqref{eq:1}. 
}
{It is clear that}
$
{f^{sos}_d} \leq {f^{sos}_{d+1}}\leq f_*$  {and $f^{mom}_d\leq f_{d+1}^{mom}\leq f_*$} for all $ {d\geq d_0}$. 
The relaxation{s~\eqref{eq:13}-\eqref{eq:d-mom} are referred to as the} \textit{hierarchy of {moment-SOS} relaxations} {for the polynomial optimization problem~\eqref{eq:1}}.    {They} can be written equivalently as {a pair of primal and dual} SDP problem{s}~(\cite{Lasserre2001}),  in which the maximal size of the positive semidefinite matrices is   $\binom{n+d}{d}\times \binom{n+d}{d} $. 
{The weak duality $f_d^{sos}\leq f^{mom}_{d}$ holds for any $d\geq d_0$.  Various sufficient conditions for the strong duality $f_d^{sos}= f^{mom}_{d}$ to hold have been proposed in the literature, including the non-empty interior condition~(\cite{Lasserre2001}), and the inclusion of the ball constraint~(\cite{JoszHenrion}). }
Even if ${f^{sos}_d}$ is finite {and the strong duality holds},  it may not be achievable.  
A classical  ``bad" example is constructed by considering the case when $n=k=1$, $m=0$, $g_1(x)=x^2$ and $f(x)=x$. In this case, ${f^{sos}_1=f^{mom}_1=}f_*=0$ but ${f^{sos}_1}$ is not achievable. 
See~\cite[Remark 1]{Parrilo02anexplicit}. { We refer to~\cite{baldi2022} and~\cite{BALDI2024102403} for a collection of examples that present possible pathological behaviors in the hierarchy of moment-SOS relaxations.} 

{The exactness of the SOS relaxation~\eqref{eq:13} is related with the problem of representing nonnegative polynomials with sum of squares, while the exactness of the moment relaxation~\eqref{eq:d-mom} is related with the problem of characterizing truncated multi-sequences having representing measures. }
In this paper, we consider the \textit{effective degree bound} problem for {the exactness of} Lasserre's {moment-}SOS hierarchy of relaxations. {This} concerns establishing an effective bound on $d$ such that 
\begin{equation}\label{a:esdgger}
\begin{aligned}
&f-f_* \in \I{g_1,\ldots,g_k}_{2d}+ \qm{h_0,\ldots,h_m}_{2d} \, ,
\end{aligned}
\end{equation}
 {which} is the same as requiring that the optimal value ${f^{sos}_d}$ of the SOS relaxation problem~\eqref{eq:13} is equal to $f_*$ and the optimal value {$f^{sos}_d$}  of~\eqref{eq:13} {can be achieved}.  {On the dual side, we aim at establishing an effective bound on $d$ such that  an optimal solution of the  polynomial optimization problem~\eqref{eq:1} can be recovered from an optimal solution of the $d$th order moment relaxation~\eqref{eq:d-mom}.  }

\subsection{Related work}

We first mention the work related with the SOS relaxation~\eqref{eq:13}.
Using Putinar's Positivstellensatz (\cite{Putinar93}),~\cite{Lasserre2001} proved the \textit{asymptotic convergence} of Lasserre's hierarchy, i.e.  
$$
\lim_{d\rightarrow +\infty} {f^{sos}_d} =f_*,
$$
under the Archimedean condition, see also~\cite{Schweighofer05}.
The \textit{finite convergence} property {of the SOS relaxation~\eqref{eq:13}}, i.e., the existence of an integer $d$ such that 
\begin{equation}\label{a:wddf} 
{f^{sos}_d}=f_*
\end{equation}
has attracted a lot of attention. \cite{Laurent07} showed that~\eqref{a:wddf}  holds for some integer $d$  when the complex variety $\{z\in \CC^n: g_1(z)=\cdots=g_k(z)=0\}$ has finitely many points.  \cite{NieReal} established that~\eqref{a:wddf} holds for some integer $d$ when the real variety $\{x\in \RR^n: g_1(x)=\cdots=g_k(x)=0\}$ has finitely many points. 
For polynomial optimization over compact semi-algebraic set, \cite{Nie14} showed that~\eqref{a:wddf}  holds for some  integer $d$ when the linear constraint qualification condition, second-order sufficient optimality condition and strict complementary slackness condition  hold at any optimal solution point. These conditions are  \textit{generic}, which means that they hold except on a Lebesgue measure zero in the space of coefficients of input polynomials.  Finite convergence for gradient/KKT type SOS relaxation has also been established in a series of papers~(\cite{NieDemmelSturmfels06,DEMMEL2007189,Nie13Jacobian,nie2019tight}), under conditions which hold generically. 
There are also many works devoted to the estimation of the gap $f_*- {f^{sos}_d}$, see~\cite{NIESchweighofer,KORDA20171,Klerk2019, FangFawzi21,Mai2022}, {\cite{Slot22,Baldi23,BaldiSlot24,laurent2024convergenceratessumssquares}}.   

Note that~\eqref{a:esdgger} is stronger than~\eqref{a:wddf}. \cite{Parrilo02anexplicit}
proved that~\eqref{a:esdgger} holds for some $d$  if $\<g_1,\ldots,g_k>$ is zero-dimensional and radical.   \cite{Marshall06} showed that~\eqref{a:esdgger} holds when $\<g_1,\ldots,g_k>$ is {real} radical and the boundary hessian condition holds.  Explicit degree bound for~\eqref{a:esdgger} to hold  is only known in special circumstances.  The first effective degree bound for~\eqref{a:esdgger} was established for the special case when the  feasible region is a grid, first by~\cite{Lasserre02} and later improved by~\cite{Laurent07} (see~\Cref{subsec:cwp} for more details).  In the  special case of binary optimization,~\cite{Fawzi} 
proved that~\eqref{a:esdgger} holds for any $d\geq \lceil{n/2} \rceil$ if $f$ is quadratic and homogeneous. 
This bound  has been extended to $d\geq \lceil{(n+r-1)/2}\rceil$ for any polynomial $f$ of degree less than or equal to $r$ and to $d\geq \lceil{(n+r-2)/2}\rceil$ 
if in addition $f$ only contains monomials of even degree, by~\cite{SakaueTakeda17}.  

{In the study of the moment relaxation \eqref{eq:d-mom}, research has focused mainly on the flat extension condition (\cite{Laurent04}) and the flat truncation condition (\cite{Nie2013}). These conditions are crucial as they not only confirm the finite convergence of the moment relaxation, i.e., $f^{mom}_d=f_*$, but also enable the extraction of minimizers from the moment matrix associated with the optimal solution $\Lambda$. \cite{LasserreLaurentRostalski08} explored cases where there are no inequality constraints and the real variety $\{x\in \R^n: g_1(x)=\cdots=g_k(x)=0\}$ is finite. They showed that the flat truncation condition holds and that the generators of the real radical ideal of $\langle g_1,\ldots,g_k \rangle$ can be computed from the kernel of the moment matrix associated with $\Lambda$ in the relative interior of $\{\Lambda \in \cL_{d}: \<\Lambda, h_0>=1\}$, for sufficiently large $d$. Similar results are suggested to apply in cases with inequality constraints, as noted by \cite{Laurent2012}. \cite{Nie2013} demonstrated the equivalence between finite convergence and the flat truncation property under certain generic conditions, assuming the optimal solution of \eqref{eq:13} is achievable and strong duality holds. More recently, \cite{BALDI2024102403} established that flat truncation occurs if and only if the support of the quadratic module associated with the minimizers is of dimension zero. Furthermore, \cite[Theorem 3.19]{BALDI2024102403} revealed the strong connection between the order at which the flat truncation condition occurs and the finite convergence degree. 
}

\subsection{Contribution}
Throughout the paper we make the following assumption. 
\begin{ass}\label{ass:ginfty}
There are at least $n$  equality constraints ($k\geq n$) and the equations
$$
 g_1^{\infty}(z)=\cdots=g_n^{\infty}(z)=0
$$
{have} no solution in $\CC^n \backslash \{\bf{0}\}$.\end{ass}
\noindent Here $g_i^{\infty}$ denotes the highest degree homogeneous part of $g_i$, see~\Cref{sec:pre} for a more precise definition. 
Our main result is stated below. 
\begin{theorem}\label{thm:main}
Under~\Cref{ass:ginfty}, if for any optimal solution $x^*$ of~\eqref{eq:1} we have
\begin{equation}\label{a:aradf}
\rank\left (  \left[\begin{array}{lll}\nabla g_1(x^*) & \cdots & \nabla g_k (x^*)\end{array} \right ]\right)=n, 
\end{equation}
then
there exist $\lambda_1,\ldots,\lambda_k,q_1,\ldots,q_{m+2}\in \R[x_1,\ldots,x_n]$ such that
$$
f-f_*= q_1^2+q_2^2+\sum_{i=1}^k \lambda_i g_i+\sum_{j=1}^m q^2_{j+2} h_j,
$$
and  
\begin{equation}\label{a:degd0}
\begin{aligned}
&\max\left( \deg(q^2_1+q^2_2),\max_{i=1,\ldots,k}\deg(\lambda_i g_i),\max_{j=1,\ldots,m} \deg(q^2_{j+2}h_j) \right)  
\\ &\leq \max\left( \deg(f), \fn+d_g, 2\fn+d_h\right).
\end{aligned}
\end{equation}
\end{theorem} 
In the above,
\begin{equation}\label{a:fn}\fn:=\sum_{i=1}^n \deg({g_i})-n{\, ,}
	\end{equation}
    \begin{equation}\label{a:dg0}
d_g:=\left\{\begin{array}{ll}
\max(\deg(g_{n+1})\, ,\ldots,\deg(g_k)) & \enspace \mathrm{if~} k\geq n+1\\
 0 &\enspace \mathrm{else}
\end{array}\right. \, ,
\end{equation}
and
\begin{equation}\label{a:dh}
d_{h}:=\left\{\begin{array}{ll}
\max(\deg(h_{1}),\ldots,\deg(h_m)) & \enspace \mathrm{if~} m\geq 1\\
 0 &\enspace \mathrm{else}
\end{array}\right. \,\, .
\end{equation}
\Cref{thm:main} implies that under~\Cref{ass:ginfty} and the rank condition~\eqref{a:aradf}, the condition~\eqref{a:esdgger} holds for some 
\begin{equation}\label{a:degd}
d\leq \lceil\max\left( \deg(f)/2, (\fn+d_g)/2, \fn+d_h/2\right) \rceil.
\end{equation} 
The rank condition~\eqref{a:aradf} in~\Cref{thm:main} is weaker than the radical condition imposed in~\cite{Parrilo02anexplicit}. 
We achieved this using the structure theorem for Artinian ring. A byproduct of our approach is a control on the number of squares in~\eqref{a:esdgger}: we proved that under the rank condition~\eqref{a:aradf}, ~\eqref{a:esdgger} can hold with $\sigma_0$ being the sum of at most two squares.

{In~\Cref{thm:flatorder}, we deduce the same type of effective degree bound for the flat truncation condition under~\Cref{ass:ginfty}, by applying the duality connection in \cite[Theorem 3.19]{BALDI2024102403}. In~\Cref{thm:realradical}, under~\Cref{ass:ginfty} we establish an effective version of the results in~\cite{LasserreLaurentRostalski08} for semidefinite characterization of real radical ideal.}

 The degree bound~\eqref{a:degd} is an explicit expression which only depends  on  the degrees of the defining polynomials $\deg(f),\deg(g_1),\ldots,\deg(g_k),\deg(h_1),\ldots,\deg(h_m)$. 
\Cref{ass:ginfty}  appears naturally in the problem of optimization over critical locus (see~\eqref{eq:gradient}). In particular, applying~\Cref{thm:main} to~\eqref{eq:gradient}, we establish the first explicit degree bound for SOS relaxation over gradient ideal, under the same conditions as in~\cite{Nie13Jacobian} for ensuring the existence of $d$ such that~\eqref{a:esdgger} or~\eqref{a:wddf} holds. 

\Cref{ass:ginfty} is an open condition which holds generically, see~\Cref{a:gofa} for further discussion. When there are exactly $n$ equalities and there is no inequality constraint (i.e., $k=n$ and $m=0$), the degree bound~\eqref{a:degd} reduces to: 
$$
 \left\lceil \max\left(\left(\sum_{i=1}^n \deg(g_i)-n\right),\deg(f)/2\right) \right \rceil.
$$
Prior to our work, the same effective degree bound was established for the grid case {by~\cite{Laurent07}}.  Our work shows that the same degree bound actually holds for a much larger class of problems. Indeed, fixing the degrees of $\{g_1,\ldots,g_n\}$ to be $r$. Then the equivalent class of grid problems has dimension $O(nr)$, while the equivalent class of problems satisfying~\Cref{ass:ginfty} has dimension $O(n^r)$, see~\Cref{subsec:cwp} for more details.

\subsection{Organization}

 This paper is organized as follows. In~\Cref{sec:pre} we fix  notations and present some results in commutative algebra that are needed in the paper. In~\Cref{sec:square} we give the nonsingularity condition which guarantees that~\eqref{a:esdgger} holds for some $d$. 
 In~\Cref{sec:cia} we establish the effective degree bound~\eqref{a:degd} under~\Cref{ass:ginfty}. {In~\Cref{sec:ddual} we present analogue degree bound results for the dual moment problem.} In~\Cref{sec:gsos} we apply the main result to gradient type SOS. In~\Cref{sec:cwp} we provide some examples satisfying~\Cref{ass:ginfty} and make some discussions on this assumption. In~\Cref{sec:perspec} we conclude and provide some perspectives on future research.  In {the} appendix, we present some deferred proofs.

\section{Notations and preliminaries}\label{sec:pre}

In this section, we fix notations and state two theorems in commutative algebra that are needed for deriving the main results of the paper. For reference of materials in this section, see~\cite{CoxLittleOshea15}.

 Denote $[n]:=\{1,\ldots,n\}$. 
We denote by $\R[x_1,\ldots, x_n]$ (resp. $\CC[x_1,\ldots,x_n]$) 
the ring of polynomials in variables $(x_1,\ldots,x_n)$ with coefficients in $\R$ 
(resp. $\CC$)
. 
For an ideal $I$ of $\R[x_1,\ldots,x_n]$ and $p,q\in \R[x_1,\ldots,x_n]$, we say that $p=q$ in the quotient ring $\R[x_1,\ldots,x_n]/I$  if $p-q\in I$.



For {a} polynomial $p\in \R[x_1,\ldots,x_n]$, we denote by  $\deg(p)$ the degree of  $p$. 
We write $\ol{p}\in \R[x_0,\ldots,x_n]$ as  the \textit{homogenization} of $p\in \R[x_1,\ldots,x_n]$:
$$
\ol{p}\left(x_0,\ldots,x_n\right):=x_0^{\deg(p)} p\left(\frac{x_1}{x_0},\ldots, \frac{x_n}{x_0}\right).
$$
And we denote by $p^{\infty}\in \R[x_1,\ldots,x_n]$ the homogeneous polynomial obtained by letting $x_0=0$ in $\bar p$, i.e.,
$$
p^{\infty}(x_1,\ldots,x_n):=\bar p(0,x_1,\ldots,x_n).
$$
Note that $p^{\infty}$ is just the sum of all those terms of $p$ with degree equal to $\deg(p)$.  For example, if $p(x_1,x_2)=x^2_1+5x_1x_2+3x_1-x_2$, then $p^{\infty}(x_1,x_2)=x^2_1+5x_1x_2$.  

We say that $\{g_1,\ldots,g_n\}$ forms an \textit{H-basis} of the ideal $\<g_1,\ldots,g_n>$ if for any $p\in \<g_1,\ldots,g_n> $, there exists $\lambda_1,\ldots,\lambda_n\in \R[x_1,\ldots,x_n]$ such that 
$
p=\sum_{i=1}^n \lambda_i g_i
$
and 
$$
\deg(\lambda_i g_i)\leq \deg(p),\enspace \forall i=1,\ldots,n.
$$ {A similar concept is that of \textit{Gr\"obner basis}, which is defined with respect to a monomial ordering, see~\cite[Chapter 2]{CoxLittleOshea15} for a precise definition.  A Gr\"obner basis with respect to a monomial ordering that is compatible with the degree is an H-basis. The converse is not necessarily true. In~\Cref{ex:cgrob}, we provide an example of H-basis which is not a Gr\"obner basis with respect to any monomial ordering. }

 Let $\bS:=\R[x_0,\ldots,x_n]$ be the ring of polynomials of $n+1$ variables with real coefficients. Denote by $\bS_d$ the subset of $\bS$ containing all the homogeneous polynomials of 
 degree $d$. Then $\bS_d$ is an $\R$-vector space of dimension $\binom{n+d}{n}$. 
 {For  a homogeneous ideal $\bI$ of $\bS$, let
 $\bI_d$ be the  subspace of $\bS_d$ given by $\bI \cap \bS_d$.} Denote by $\bS_d/\bI_d$ the quotient  space.  For any $\R$-vector space $W$, we denote by $\dim_{\R} W$ its dimension as an $\R$-vector space.
We now state two theorems that are essential for deriving the effective degree bound.
\begin{theorem} [B\'ezout's Theorem]\label{thm:Bezout}
Under~\Cref{ass:ginfty}, we have
\begin{equation}\label{a:dimension0}\dim_{\R} \left( \RR[x_1,\ldots,x_n]/ \langle g_1,\ldots, g_n\rangle \right)=\prod_{i=1}^n \deg(g_i).\end{equation}
\end{theorem}
For reference of B\'ezout's Theorem, see~\cite[Theorem 5.5, \S 3.5]{Coxusingag05}.
Without further specification, $\bI$ denotes the ideal of $\bS$ generated by  the homogeneous polynomials $\bar g_1,\ldots,\bar g_n$.
\begin{theorem}[see~\cite{MollerHbasis}]\label{thm:dimSI}
Under~\Cref{ass:ginfty}, $\{g_1,\ldots,g_n\}$ forms an H-basis of the ideal $\<g_1,\ldots,g_n>$ and
\begin{equation}\label{a:facd}
\dim_{\RR} \bS_{\fn}/\bI_{\fn}= \prod_{i=1}^n \deg(g_i),
\end{equation}
where $\fn$ is the constant defined in~\eqref{a:fn}.
\end{theorem}
\cite{MollerHbasis} presented a sketched proof for the results stated in~\Cref{thm:dimSI}. In~\Cref{seca:thmdim}, we provide a complete proof for the sake of thoroughness.

\section{Square and sum of {two} squares representation of a nonnegative polynomial over finite variety}\label{sec:square}

In this section we present results on  sum of squares representation of nonnegative polynomial over finite variety.
We give  nonsingularity conditions under which  any nonnegative polynomial over a finite variety can be written as one single square  or as the sum of at most two squares in the   ideal. 

Denote by  $V_{\CC}(g_1,\ldots,g_k)$ the complex variety where $g_1,\ldots,g_k$  all vanish: $$V_{\CC}(g_1,\ldots,g_k):=\{z\in \CC^n: g_1(z)=\cdots=g_k(z)=0\}.$$
Equation~\eqref{a:dimension0} implies  that $\dim_{\R} \left( \RR[x_1,\ldots,x_n]/ \langle g_1,\ldots, g_k\rangle \right)<+\infty$ and thus  the complex variety $V_{\CC}(g_1,\ldots,g_k)$ is finite. 
Let $v\in V_{\CC}(g_1,\ldots,g_k)$. If the matrix 
\begin{equation}\label{a:sse}
\left[\begin{array}{ccc}
    \nabla g_1(v) & \cdots & \nabla g_k(v)  
\end{array}\right]
\end{equation}
has rank $n$ in $\CC^n$, then $v$ is said to be a \textit{nonsingular solution}.  Otherwise, $v$ is a \textit{singular} solution.  We denote by 
$$
\cS_{\CC}(g_1,\ldots,g_k):=\left\{v\in V_{\CC}(g_1,\ldots,g_k): \rank \left[\begin{array}{ccc}
    \nabla g_1(v) & \cdots & \nabla g_k(v)  
\end{array}\right] <n \right\}
$$
the set of  singular solutions, and by 
$$
\cS_{\RR}(g_1,\ldots,g_k):= \R^n \cap \cS_{\CC}(g_1,\ldots,g_k)
$$
the set of singular solutions that are real. 

We start by stating several technical lemmas. For the sake of clarity in our presentation, their proofs are deferred to the appendix.
Recall that the maximal ideals of $\CC[x_1,\ldots,x_n]$ are in one-to-one correspondence with the points in $\CC^n$. Every maximal ideal of $\RR[x_1,\ldots,x_n]$ can be written as the intersection of $\RR[x_1,\ldots,x_n]$ with a maximal ideal of $\CC[x_1,\ldots,x_n]$. Hence every point $v\in \CC^n $ is associated with a maximal ideal of $\R[x_1,\ldots,x_n]$.   Let $I$ be an ideal of $\R[x_1,\ldots,x_n]$. We denote by $I^{\ell}$ the ideal generated by all products $a_1a_2\cdots a_{\ell}$ in which each factor $a_i$ belongs to $I$.
\begin{lemma}\label{l:nonsingular}
Let $\fm$ be the maximal ideal of $\R[x_1,\ldots,x_n]$ associated with a solution $v\in V_{\CC}(g_1,\ldots,g_k)$. If $v$ is nonsingular, then 
$$
\fm=\fm^\ell+\<g_1,\ldots,g_k>,\enspace \forall \ell\geq 1.
$$
\end{lemma}
\begin{lemma}\label{l:squarerootlemma01}
Let $\fm$ be the maximal ideal of $\RR[x_1,\ldots,x_n]$ associated with a point  $v\in \CC^n$.  
 Let  $p\in \RR[x_1,\ldots,x_n]$. If 
\begin{enumerate}
\item[$(1)$]   $p(v)\neq 0$, and
\item[$(2)$]  $p(v)>0$ if $v\in \R^n$,
\end{enumerate}
then for any $\ell\geq 1$,  there is $q\in \R[x_1,\ldots,x_n]$ such that 
$p=q^2$ in $\RR[x_1,\ldots,x_n]/{\fm^\ell}$.
\end{lemma} 
\begin{lemma}\label{l:squarerootlemma-1}
Let $\fm$ be a maximal ideal of $\RR[x_1,\ldots,x_n]$ associated with a non-real point $v\in \CC^n \backslash \R^n$. For any $\ell \geq 1$, there is $q\in \R[x_1,\ldots,x_n]$ such that 
$q^2\equiv -1$ in $\RR[x_1,\ldots,x_n]/{\fm^\ell}$.
\end{lemma}
 Using the above three lemmas, we prove the main theorem on square and sum of two squares representation of a nonnegative polynomial over finite variety.   Recall that the maximal ideals of $\R[x_1,\ldots,x_n]/\langle g_1,\ldots,g_k\rangle$  are the maximal ideals of $\RR[x_1,\ldots,x_n]$  which contain the ideal $\langle g_1,\ldots,g_k\rangle$. 
\begin{theorem}\label{thm:squareroot}
Suppose that $V_{\CC}(g_1,\ldots,g_k)$ is a finite set. 
Let $p\in \R[x_1,\ldots,x_n]$  such that 
$
p(v)\geq 0$ for all $v\in V_{\CC}(g_1,\ldots,g_k)\cap \R^n.
$
\begin{enumerate}
\item[(a)] If 
 $p(v)\neq 0 $ for all $v\in \cS_{\CC}(g_1,\ldots,g_k)$, 
then there exist $q,\lambda_1,\ldots,\lambda_k\in \R[x_1,\ldots,x_n]$ such that 
$$
p=q^2+\sum_{i=1}^k \lambda_i g_i.
$$
\item[(b)] If $p(v)\neq 0$ for all  $ v\in \cS_{\R}(g_1,\ldots,g_k)$,
then there exist $q_1,q_2,\lambda_1,\ldots,\lambda_k\in \R[x_1,\ldots,x_n]$ such that 
$$
p=q_1^2+q_2^2+\sum_{i=1}^k \lambda_i g_i.
$$
\end{enumerate}
\end{theorem}
\begin{proof}
 For ease of notation, denote $A:=\R[x_1,\ldots,x_n]/\<g_1,\ldots,g_k>$. 
 Since $\dim_{\RR}(A)<+\infty$,  $A$ is a finite-dimensional $\RR$-vector space and thus  $A$ is Artitian~(\cite[Exercise 3, Chapter 8]{AtiyahMac}). It follows that $A$ only has finitely many maximal ideals~(\cite[Corollary 8.2]{AtiyahMac}).  
Let $\fm_1,\ldots,\fm_s$ be the maximal ideals of $A$. 
  By the structure theorem for Artinian rings~(\cite[Theorem 8.7]{AtiyahMac}), there is $\ell\in \NN$ such that
 \begin{equation}\label{a:artindecomp}
 A\simeq \prod_{i=1}^s A/{\fm}^\ell_i,
 \end{equation}
 with the ring isomorphism $\pi$ being the product of the canonical morphism $\pi_i:A\to A/{\fm}^\ell_i$.

   We first show (a). Assume that  $p(v)\neq 0$ for all $v\in \cS_{\CC}(g_1,\ldots,g_k)$. We shall prove that for each $i\in [s]$, $p$ has a square root $q_i$ in $A/{\fm^\ell_i}$.
   Take any $i\in [s]$, and let $v\in V_{\CC}(g_1,\ldots,g_k)$ be a solution associated with $\fm_i$.  Suppose that $v$ is a nonsingular point.  By~\Cref{l:nonsingular}, we have $A/\fm_i=A/\fm^\ell_i$. 
 Since $p(v)\geq 0$ if $v\in \R^n$, there is $q_i\in \RR[x_1,\ldots,x_n]$ such that $q_i^2=p$ in $\RR[x_1,\ldots,x_n]/\fm_i$. It follows that  $q_i^2=p$ in $A/\fm_i$ and consequently $q_i^2=p$ in $A/\fm^\ell_i$.
 If $v$ is a singular point, we have $p(v)\neq 0$, and $p(v)>0$ if $v$ is  real. Hence the two conditions in~\Cref{l:squarerootlemma01}  are satisfied. We then deduce the existence of $q_i\in \RR[x_1,\ldots,x_n]$  
 such that $q_i^2=p$ in $\RR[x_1,\ldots,x_n]/\fm^\ell_i$ and thus  $q_i^2=p$ in $A/\fm_i^\ell$. We thus proved the existence of $q_1,\ldots,q_s$ such that 
 $$
 \pi(p)=\left( \pi_1(q_1^2),\ldots, \pi_s(q^2_s)\right).
 $$
Since $\pi$ is surjective, there is $q\in A$ such that $\pi(q)=\left( \pi_1( q_1),\ldots,  \pi_s(q_s)\right)$. It follows that 
$\pi(q^2)=\left( \pi_1(q_1^2),\ldots, \pi_s(q^2_s)\right)=\pi(p)$. Since $\pi$ is injective, we have $p=q^2$ in $A$ and thus (a) is proved.

To prove (b),  we can recycle most of the arguments above  except when $v$ is a non-real singular point associated with $\fm_i$.  In this case  we know from~\Cref{l:squarerootlemma-1} that there exist $q\in \R[x_1,\ldots,x_n]$ such that 
$q^2=-1$ in $\R[x_1,\ldots,x_n]/\fm_i^\ell$. Then
$
p=(p+1/4)^2+q^2(p-1/4)^2
$ in $\R[x_1,\ldots,x_n]/\fm_i^\ell$. Thus there exists $q_{i,1}, q_{i,2}\in \R[x_1,\ldots,x_n]$ such that $q_{i,1}^2+q_{i,2}^2=p$ in $A/\fm_i^\ell$. Hence there exist $q_{1,1},\ldots,q_{s,1}, q_{1,2},\ldots,q_{s,2}$ such that 
$$
 \pi(p)=\left( \pi_1(q_{1,1}^2)+\pi_1(q_{1,2}^2),\ldots, \pi_s(q^2_{s,1})+\pi_s(q_{s,2}^2)\right).
 $$
Since $\pi$ is surjective, there is $q_1,q_2\in A$  such that $$\pi(q_1)=\left( \pi_1( q_{1,1}),\ldots,  \pi_s(q_{s,1})\right),$$ and $$\pi(q_2)=\left( \pi_1( q_{1,2}),\ldots,  \pi_s(q_{s,2})\right).$$ It follows that 
$$\pi(q_1^2+q_2^2)=\left( \pi_1(q_{1,1}^2)+\pi_1(q_{1,2}^2),\ldots, \pi_s(q^2_{s,1})+\pi_s(q_{s,2}^2)\right)=\pi(p).$$ Since $\pi$ is injective, we have $p=q_1^2+q_2^2$ in $A$.

 \end{proof}

We next apply~\Cref{thm:squareroot} for determining the membership   of $f-f_*$ in the quadratic module $\<g_1,\ldots,g_k>+\qm{h_0,\ldots,h_m}${, as well as bounding the number of squares.}
\begin{corollary}\label{coro:ffstar}Suppose that $V_{\CC}(g_1,\ldots,g_k)$ is a finite set. 
If for {every} optimal solution $v$ of~\eqref{eq:1} we have
\begin{equation}\label{a:aradf1}
\rank\left (  \left[\begin{array}{lll}\nabla g_1(v) & \cdots & \nabla g_k (v)\end{array} \right ]\right)=n, 
\end{equation}
then { there exist $\lambda_1,\ldots\lambda_k,q_1,\ldots,q_{m+2}\in \R[x_1,\ldots,x_n]$ such that
\begin{equation}\label{eq:finmodule2q}
f-f_*=q_1^2+q_2^2+\sum_{i=1}^k \lambda_i g_i+\sum_{j=1}^m q_{j+2}^2 h_j.
\end{equation}
}
\end{corollary}
\begin{proof}
Let $\{v_1,\ldots, v_s\}$ be all the points in $V_{\CC}(g_1,\ldots,g_k)\cap \R^n$. W.l.o.g we assume that there is an integer  $t\in [s]$  such that
$$
\{v_1,\ldots,v_{t}\}=\{x\in \R^n: g_1(x)=0,\ldots,g_k(x)=0, h_1(x)\geq 0,\ldots,h_m(x)\geq 0\}.
$$ 
For each $i \in [s]$ such that $i\geq t+1$, let $\phi_i\in \R[x_1,\ldots,x_n]$ be a polynomial satisfying 
$$
\phi_i(v_j)=\left\{\begin{array}{ll} 0 & \mathrm{~} j \in [s], j\neq i\\
 1 &  \mathrm{~} j=i
\end{array}\right. \enspace {,}
$$ and let  $\xi_i\in [m]$ such that $h_{\xi_i}(v_i)<0$. Consider the following polynomial
$$
p:=f-f_* -  \sum_{i=t+1}^s\frac{\mathbf{1}_{f(v_i)\leq f_*}\left( f(v_i)-f_*-1\right)}{h_{\xi_i}(v_i)} \phi^2_i h_{\xi_i}.
$$
It can be checked that
$$
p(v_i)=\left\{\begin{array}{ll}f(v_i)-f_* & \enspace i\leq t \mathrm{~or~}  f(v_i)>f_*\\
 1  & \enspace t+1\leq i\leq s \mathrm{~and~} f(v_i)\leq f_* 
\end{array}\right.\enspace {.}
$$
Then $p(v_i)\geq 0$ for all $i\in [s]$. In addition, if $p(v_i)=0$, then $v_i$ must be an optimal solution of~\eqref{eq:1} and hence $v_i\notin \cS_{\R}(g_1,\ldots,g_k)$. By~\Cref{thm:squareroot}, there exist $q_1,q_2,\lambda_1,\ldots,\lambda_k\in \R[x_1,\ldots,x_n]$ such that 
$$
p=q_1^2+q_2^2+\sum_{i=1}^k \lambda_i g_i,
$$
which can be written as 
$$
f-f_* = \sum_{i=1}^k \lambda_i g_i+ q_1^2+q_2^2 +\sum_{i=t+1}^s\frac{\mathbf{1}_{f(v_i)\leq f_*}\left( f(v_i)-f_*-1\right)}{h_{\xi_i}(v_i)} \phi^2_i h_{\xi_i}
$$
Note that $\frac{\mathbf{1}_{f(v_i)\leq f_*}\left( f(v_i)-f_*-1\right)}{h_{\xi_i}(v_i)} $ is nonnegative. We then deduced~\eqref{eq:finmodule2q}.

 \end{proof}

{
An immediate consequence of~\Cref{coro:ffstar} is the following result on the inclusion of $f-f_*$ in the quadratic module.  
\begin{corollary}\label{coro:ffstar2}Under the same conditions as~\Cref{coro:ffstar}, we have
\begin{equation}\label{f:ersdfer}
f-f_*\in \langle g_1,\ldots,g_k\rangle+\qm{h_0,\ldots,h_m}.
\end{equation}
\end{corollary}
\begin{remark}\label{rem:c923M}
\Cref{coro:ffstar2} could also be proved by applying~\cite[Theorem 9.2.3]{Marshallbook} (see also~\cite[Theorem 2.8]{SCHEIDERER2005558}), that we recall below for clarity  of discussion. 
\begin{theorem}\label{thm:Marshall}\cite{Marshallbook}
Suppose $M$ is an Archimedean quadratic module of a commutative ring $A$ containing $1/2$, $p\in A$, $p\geq 0$ on $\cK_M$ and $A/J$ is Artinean, where $J:=(M+(p^2))\cap -(M+(p)^2)$. Then the following are equivalent:
\begin{itemize}
\item[(1)] $p\in M$.
\item[(2)] For each zero $\alpha$ of $p$ in $\cK_M$, $p$ belongs to the quadratic module of $\hat A_{\alpha}$ generated by the image of M.
\end{itemize}
\end{theorem}
We refer to~\cite{Marshallbook} for the precise meaning of the notions in the above theorem. Translating in the language of this paper,  $A$ corresponds to $\R[x_1,\ldots,x_n]$, $M$ corresponds to  $\langle g_1,\ldots,g_k\rangle+\qm{h_0,\ldots,h_m}$, $\cK_M$ corresponds to the feasible region of~\eqref{eq:1}, $p$ corresponds to $f-f_*$ and the condition "\textit{$\R[x_1,\ldots,x_n]/J$ is Artinean}" is equivalent to that the complexity variety $\{z\in \CC^n: q(z)=0,\enspace \forall q\in J\}$ is a finite set, which holds in the setting of this paper because $\langle g_1,\ldots,g_k\rangle \subset J$ and $V_{\CC}(g_1,\ldots,g_k)$ is a finite set. Hence to obtain~\eqref{f:ersdfer} it remains to show that the condition (2) in the above theorem holds, which essentially boils down to the proof of~\Cref{l:nonsingular} and~\Cref{l:squarerootlemma01}.

In fact, similar to our proof of~\Cref{thm:squareroot},  the proof of~\cite[Theorem 9.2.3]{Marshallbook} also relies on the fundamental structure theorem of Artin rings. The key difference 
is that the ring isomorphism~\eqref{a:artindecomp} is applied to the quotient ring $A:=\R[x_1,\ldots,x_n]/J$ in~\cite[Theorem 9.2.3]{Marshallbook}, rather than to $A:=\R[x_1,\ldots,x_n]/\langle g_1,\ldots,g_k\rangle$ as  in our proof of~\Cref{thm:squareroot}.  Since  the radical of the ideal $J$ contains  the real radical of the ideal $\langle g_1,\ldots, g_k\rangle$~(\cite[P.23]{Marshallbook}),  $J$ is zero-dimensional if the real variety $V_{\CC}(g_1,\ldots,g_k)\cap \R^n$ is a finite set. Hence 
 our setting is less general than that considered in~\cite[Theorem 9.2.3]{Marshallbook}. Despite the loss of generality, the finite complex variety condition allows us to apply the fundamental structure theorem of Artinian rings~\eqref{a:artindecomp} directly on the quotient ring $\R[x_1,\ldots,x_n]/\langle 
 g_1,\ldots,g_k\rangle$. This results in a more straightforward proof and a stronger outcome, as stated in~\Cref{thm:squareroot}, which shows that $p$ is a sum of at most two squares in $\<g_1,\ldots,g_k>$. In contrast, the approach of~\cite[Theorem 9.2.3]{Marshallbook} only yields that $p$ is a square in $J$, and therefore can not lead to any control on the number of squares.
\end{remark}
\begin{remark}
When the ideal $\langle g_1,\ldots,g_k \rangle$ is both zero-dimensional and radical, all the solutions in $V_{\CC}(g_1,\ldots,g_k)$ are nonsingular. For this special case, we can replace~\eqref{eq:finmodule2q} by
\begin{equation}\label{eq:finmodule1q}
f-f_*=q_1^2+\sum_{i=1}^k \lambda_i g_i+\sum_{j=1}^m q_{j+1}^2 h_j.
\end{equation}
The same type of SOS representation of nonnegative polynomials over finite varieties have been constructed previously in~\cite{Parrilo02anexplicit} (see also~\cite[Theorem 11]{Laurent07}). In fact, due to the radical ideal condition, the multiplicity order $\ell$ in the isomorphism~\eqref{a:artindecomp} is equal to 1, and~\cite{Parrilo02anexplicit} actually constructed explicitly this isomorphism using the interpolation polynomials. 
\end{remark}
\begin{remark}
\cite[Theorem 10.4.1]{Marshallbook} highlights the critical importance of the closedness of the truncated quadratic module $\mQ_{2d}$ in ensuring strong duality within the moment-SOS hierarchy \eqref{eq:13}-\eqref{eq:d-mom}. In the absence of inequality constraints, if the ideal $\<g_1,\ldots,g_k>$ is real radical and $\{g_1,\ldots,g_k\}$ form an 
H-basis, \cite[Lemma 4.1.4]{Marshall_2009} shows that the truncated quadratic modules 
$\mQ_{2d}$
  are closed for each $d$. Consequently, using the same reasoning as in \cite[Theorem 10.4.1]{Marshallbook}, if $f^{sos}_{d}\in \R$, then the strong duality holds and the optimal solution of \eqref{eq:13} is achieved. In this case,  the finite convergence of the moment hierarchy $f^{mom}_d=f_*$ leads to~\eqref{f:ersdfer}, i.e. the exactness of the SOS relaxation. For a detailed discussion and generalization, including cases with inequality constraints, we refer to \cite[Proposition 3.10]{BALDI2024102403}.
\end{remark}
}

\section{Effective degree bound}\label{sec:cia}
In this section, we derive an effective degree bound under~\Cref{ass:ginfty} and provide the proof for~\Cref{thm:main}.

We first show that there is a basis of $\R[x_1,\ldots,x_n]/\<g_1,\ldots,g_n>$ such that the maximum degree of a polynomial in the basis is less than $\fn$, defined in~\eqref{a:fn}.

\begin{prop}\label{claim1} 
Under~\Cref{ass:ginfty}, for any $p\in \R[x_1,\ldots,x_n]$,  there is $q \in \R[x_1,\ldots,x_n]$ with $\deg(q)\leq \fn$ such that 
$$
 p - q \in \<g_1,\ldots,g_n>.
$$
\end{prop}
\begin{proof}
By~\Cref{thm:Bezout} and~\Cref{thm:dimSI},
\begin{equation}\label{a:posd}
\dim_{\RR} \bS_{\fn}/\bI_{\fn}=\prod_{i=1}^n \deg(g_i)=\dim_{\RR} \RR[x_1,\ldots,x_n]/\<g_1,\ldots,g_n>.
\end{equation} Define the $\R$-linear map $\phi:  \bS_{\fn} \rightarrow \R[x_1,\ldots,x_n] $ as follows: 
$$
\phi(h)(x_1,\ldots,x_n):=h(1,x_1,\ldots,x_n).
$$
This map induces an $\R$-linear map from $ \bS_{\fn}/\bI_{\fn}$ to $\R[x_1,\ldots,x_n]/\<g_1,\ldots,g_n>$. 
We next argue that this induced $\R$-linear map is injective. 
Let $h\in \bS_{\fn}$  such that $\phi(h)\in \<g_1,\ldots,g_n>$. By~\Cref{thm:dimSI},  $\{g_1,\ldots,g_n\}$ forms an H-basis of the ideal $\<g_1,\ldots,g_n>$. Hence  there are polynomials $\lambda_1,\ldots,\lambda_n\in \R[x_1,\ldots,x_n]$ such that  $\phi(h)=\sum_{i=1}^n \lambda_i g_i$  and \begin{equation}\label{eq:lgl}\deg(\lambda_i g_i)\leq \deg(\phi(h)),\enspace \forall i\in [n].\end{equation}
We have
$$
 h\left(1,\frac{x_1}{x_0},\ldots,\frac{x_n}{x_0}\right)=\sum_{i=1}^n \lambda_i\left(\frac{x_1}{x_0},\ldots, \frac{x_n}{x_0}\right) g_i\left(\frac{x_1}{x_0},\ldots, \frac{x_n}{x_0}\right).
$$
Multiplying both sides by $x_0^{\deg(h)}$ we obtain 
$$
h\left(x_0,{x_1},\ldots,{x_n}\right)=\sum_{i=1}^n x_0^{\deg(h)-\deg(g_i)-\deg(\lambda_i)}\bar \lambda_i\left(x_0,{x_1},\ldots, {x_n}\right) \bar g_i\left(x_0,{x_1},\ldots, {x_n}\right).
$$
By~\eqref{eq:lgl}, we have  $\deg(\lambda_i)\leq \deg(\phi(h))-\deg(g_i)$ whenever $\lambda_i\neq 0$. Therefore, 
$\deg(h)-\deg(g_i)-\deg(\lambda_i)\geq 0$ when $\bar \lambda_i\neq 0$. We deduce that $h\in \bI$. Thus, the $\R$-linear map  from $ \bS_{\fn}/\bI_{\fn}$ to $\R[x_1,\ldots,x_n]/\<g_1,\ldots,g_n>$ induced by $\phi$ is injective. 
  By~\eqref{a:posd}, this map is also surjective.
Therefore, for any $p\in \R[x_1,\ldots,x_n]$,  there is $h\in \bS_{\fn}$ such that
$$
 p \in  \phi(h) +\<g_1,\ldots,g_n>.
$$
Finally note that $\deg(\phi(h))\leq \deg(h)=  \fn$.

\end{proof}

\begin{theorem}\label{thm:debound}
Under~\Cref{ass:ginfty}, if 
$$
p\in \<g_1,\ldots,g_k>+ \qm{h_0,\ldots,h_m},
$$
then 
$$
p\in \I{g_1,\ldots,g_k}_{2d}+ \qm{h_0,\ldots,h_m}_{2d},
$$
for some \begin{equation}\label{a:dles}d\leq \lceil\max\left(\deg(p)/2, (\fn+d_g)/2, \fn+d_h/2\right)\rceil.\end{equation}
\end{theorem}
\begin{proof}
Let $\lambda_1,\ldots,\lambda_k\in \R[x_1,\ldots,x_n]$ and $\sigma_0,\ldots,\sigma_m\in \Sigma [x_1,\ldots,x_n]$ such that 
$$
p=\sum_{i=1}^k \lambda_i g_i+\sum_{j=0}^m \sigma_j h_j.
$$
By~\Cref{claim1}, we can assume w.l.o.g. that $\max \left(\deg(\lambda_{n+1}),\ldots,\deg(\lambda_k)\right)\leq \fn$ and 
$\max \left(\deg(\sigma_{0}),\ldots,\deg(\sigma_m)\right)\leq 2\fn$. Then the function 
$$
\phi:=
p-\sum_{i=n+1}^{k}\lambda_i g_i-\sum_{j=0}^m \sigma_j h_j $$ is a polynomial of degree bounded by
$$
\max\left(\deg(p), \fn+d_g, 2\fn+d_h\right).
$$
Also, $\phi\in \<g_1,\ldots,g_n>$. By the H-basis property, there exist $\lambda_1,\ldots,\lambda_n\in \R[x_1,\ldots,x_n]$ such that $\deg(\lambda_i g_i) \leq {\deg(\phi)}$ and 
$$
\phi=\sum_{i=1}^n \lambda_i g_i.
$$
 \end{proof}
Our main result \Cref{thm:main} follows as an immediate consequence of~\Cref{coro:ffstar} and~\Cref{thm:debound}{, because the process outlined in the proof of~\Cref{thm:debound} does not change the number of squares in each $\sigma_i$, for $i=0,1,\ldots,m$. }
\begin{remark}
 The \textit{resultant} of $p_1,\ldots,p_n$, denoted by $\Res(p_1,\ldots,p_n)$, is a polynomial in the coefficients of $p_1,\ldots,p_n$ satisfying 
$$
\Res(p_1,\ldots,p_n)=0 \mathrm{~if~and~only~if~} \exists z\in \CC^n\backslash\{{\bf{0}}\} \mathrm{~s.t.~} p_1(z)=\cdots=p_n(z)=0.
$$
\Cref{ass:ginfty} can be stated equivalently as 
$$
\Res(g^{\infty}_1,\ldots,g^{\infty}_n)\neq 0.
$$
\end{remark}
By Putinar's Positivstellensatz, when $V_{\CC}(g_1,\ldots,g_k)$ is finite, for any $\epsilon>0$ we have
$$
f-f_*+\epsilon \in \<g_1,\ldots,g_k>+\qm{h_0,\ldots,h_m}.
$$
In view of~\Cref{thm:debound}, under~\Cref{ass:ginfty}, for any $\epsilon>0$ we have
$$
f-f_*+\epsilon \in \I{g_1,\ldots,g_k}_{2d}+\qm{h_0,\ldots,h_m}_{2d}
$$
where $d$ satisfies~\eqref{a:degd}.
Hence if only~\eqref{a:wddf} is concerned, we can drop the rank condition~\eqref{a:aradf}. 
\begin{corollary}\label{cor:main}
Under~\Cref{ass:ginfty}, we have ${f_{d}^{sos}=f_{d}^{mom}=f_*}$ for some 
\begin{equation}\label{a:dles2}d\leq \lceil\max\left(\deg(f)/2, (\fn+d_g)/2, \fn+d_h/2\right)\rceil.\end{equation}
\end{corollary}

\section{{Effective degree bound in the dual approach of moments}}\label{sec:ddual}

In this section, we present analogue degree bound results from the dual perspective, focusing on the bound on the order $d$ at which  an optimal solution of~\eqref{eq:1} can be recovered from an optimal solution of the $d$th order moment relaxation~\eqref{eq:d-mom} . 

For $\Lambda \in (\R[x_1,\ldots,x_n]_{2d})^{*}$ and $t\leq d$,
 $M_t(\Lambda)$  denotes the symmetric bilinear map: 
$$
M_t(\Lambda)(p,q)=\<\Lambda,pq>,\enspace \forall p,q \in \R[x_1,\ldots,x_n]_t.
$$
For $v\in \R^n$, denote by $[v]_d\in (\R[x_1,\ldots,x_n]_{d})^*$ the evaluation map: $\<[v]_d,p>=p(v)$. In the following,  $\cF$ denotes the feasible region of~\eqref{eq:1}, and $\cI_{\R}(\cF):=\{p\in \R[x_1,\ldots,x_n]: p(v)=0,\enspace \forall v\in \cF\}$ denotes the real vanishing ideal of $\cF$. For any $d\geq 1$, denote the truncated ideal $\cI_{\R}(\cF)_{d}:=\cI_{\R}(\cF)\cap \R[x_1,\ldots,x_n]_d$.  The \textit{interpolation degree} of $\cF$ is the minimal degree of a family of interpolation polynomials of $\cF$. Recall the definitions of $\mQ_d$ and $\cL_d$ in~\eqref{eq:Qd} and~\eqref{eq:cL}.   Denote by $\bar \mQ_d$ the closure of $\mQ_d$ with respect to the Euclidean topology in the finite dimensional vector space $\R[x_1,\ldots,x_n]_d$. Let 
$\cK_{d}:=\{\Lambda \in \cL_{d}: \<\Lambda,h_0>=1\}$.  We say that
$\Lambda$ is a generic element of $\cK_{2d}$ if $\rank(M_d(\Lambda))=\max \{\rank(M_d(\Lambda')):\Lambda'\in \cK_{2d}\}$.
\begin{definition}
For $\Lambda \in \cK_{2d}$ and $t<d$, the \textit{flat truncation condition holds at degree $t$} if 
\begin{equation}\label{eq:rankM}
\rank(M_t (\Lambda) )=\rank(M_{t+s} (\Lambda)).
\end{equation}
for some integer $0<s\leq d-t$ satisfying
\begin{equation}\label{eq:scondition}
\left\{
\begin{aligned}
&2s\geq \max_{i\in [m]} \deg(h_i), \\
&t+2s\geq  \max_{i\in [k]} \deg(g_i).
\end{aligned}
\right.
\end{equation}
\end{definition}
The flat truncation condition~(\cite{LasserreLaurentRostalski08,Nie2013}) extends the \textit{flat extension condition} proposed by \cite{CurtoFialkow96}. This condition is significant both theoretically and practically in certifying convergence and extracting optimal solutions due to the following result.
\begin{theorem}[see~\cite{Laurent04}]\label{thm:L04}
Let $\Lambda\in \cK_{2d}$. If the rank condition~\eqref{eq:rankM} holds for some $t<d$ and $0<s\leq d-t$ satisfying~\eqref{eq:scondition}, then $\Lambda_{2(t+s)}$ belongs to the convex hull of the set $\{[v]_{2(t+s)}: v\in \cF\}$.
\end{theorem}
The requirement in~\eqref{eq:scondition} differs slightly from what appears in the existing literature. Previous works either focus solely on inequality constraints (\cite{Laurent04,Nie2013,BALDI2024102403}) or treat equality constraints as two inequality constraints~(\cite{LasserreLaurentRostalski08,Laurent2012}), resulting in the condition  $2s \geq  \max \left(\max_{i\in [m]} \deg(h_i), \max_{i\in [k]} \deg(g_i)\right)$.
We have made a slight adjustment to address equality constraints separately. For completeness, a brief proof outline of \Cref{thm:L04} is provided in~\Cref{sec:proofL04} of the appendix.

\cite[Theorem 3.19]{BALDI2024102403} establishes a close connection between the order at which the flat truncation condition holds and the order of finite convergence. We present this important result below, slightly modified to align with our notations. (Recall the definition of $\dim_{\R}$ in~\Cref{sec:pre}.)
 
\begin{theorem}[\cite{BALDI2024102403}]\label{thm:BALDI}
Suppose that $\dim_{\R}{\R[x_1,\ldots,x_n]}/\left(\mQ \cap (-\mQ)\right)<+\infty$, where $\mQ:=\langle g_1,\ldots,g_k\rangle +\qm{h_0,\ldots,h_m}$.  Then $\cF$ is finite and for any $s\geq 1$, there exists $d\geq \theta+s$ such that $\cI_{\R}(\cF)_{2(\theta+s)}\subset \overline \mQ_{2d}$, where $\theta$ is the interpolation degree of $\cF$, and for any $\Lambda \in \cK_{2d}$ we have
\begin{equation}\label{eq:Mthetas}
\rank(M_\theta (\Lambda)) =\rank(M_{\theta+s} (\Lambda)).
\end{equation}
\end{theorem}
\begin{remark}
In the original formulation of~\cite[Theorem 3.19]{BALDI2024102403}, 
there are no equality constraints and
the integer $s$ is specified as $\lceil\max_{i\in [m]} \deg(h_i)/2\rceil$. This  particular choice of $s$ is connected to the first condition in requirement ~\eqref{eq:scondition}. Their method can be naturally extended to accommodate equality constraints, and the proof provided by \cite[Theorem 3.19]{BALDI2024102403} is valid for any integer $s\geq 1$. 
\end{remark}
\begin{theorem}\label{thm:flatorder}
Under~\Cref{ass:ginfty}, for every
\begin{equation}\label{eq:dlb}
d\geq  \left\lceil\max\left((\fn+d_g)/2, \fn+d_h/2, \fn+1\right)\right\rceil,
\end{equation}
and every $\Lambda \in \cK_{2d}$,
the flat truncation condition holds at some degree $\theta \leq \fn$. 
\end{theorem}
\begin{proof} 
Let any $d$ satisfy~\eqref{eq:dlb}. Let  any $p\in \cI_{\R}(\cF)_{2d}$. In view of~\Cref{thm:squareroot} (or simply by Putinar’s Positivstellensatz), 
for any $\epsilon>0$, $p+\epsilon \in \mQ$. Since $d$ satisfies~\eqref{eq:dlb},  for any $\epsilon>0$, we have
 $p+\epsilon \in \mQ_{2d}$ for all $\epsilon>0$ by~\Cref{thm:debound}. Consequently,  $p\in \overline{\mQ_{2d}}$.
We thus proved that $\cI_{\R}(\cF)_{2d}\subset   \overline{\mQ_{2d}} $. 
Let $\theta$ be the interpolation degree of $\cF$. Note that $\theta\leq \fn <d$ by~\Cref{claim1} and~\eqref{eq:dlb}.
Applying directly~\Cref{thm:BALDI}, we obtain that for each $\Lambda \in \cK_{2d}$ and $1\leq s\leq d-\theta$, the rank condition~\eqref{eq:Mthetas} holds.
Next we check that $s= d-\theta$ and $t=\theta$ satisfies the condition~\eqref{eq:scondition}, which reduces to show that
$$
\enspace 2(d-\theta)\geq    \max_{i\in [m]}\deg(h_i), \enspace 2d-\theta\geq   \max_{i\in [k]}\deg(g_i).
$$
Since $\theta\leq \fn$, it suffices to verify (recall the definition of $d_g$ and $d_h$ in~\eqref{a:dg0} and~\eqref{a:dh})
\begin{equation}\label{eq:wde}
2(d-\fn) \geq   d_h ,\enspace 2d \geq  \fn+ d_g, \enspace 2d\geq \fn+ \max_{i\in [n]}\deg(g_i).
\end{equation}
It is easy to check that $\max_{i\in [n]} \deg(g_i) \leq \fn+1$. Hence~\eqref{eq:wde} holds for all $d$ satisfying~\eqref{eq:dlb}.
\end{proof}
\Cref{thm:flatorder} provides an effective degree bound for the order at which the flat truncation condition holds under \Cref{ass:ginfty}. Together with~\Cref{thm:L04}, we obtain an effective degree bound on the relaxation order $d$ at which an optimal solution of~\eqref{eq:1} can be recovered from an optimal solution of~\eqref{eq:d-mom}.
\begin{remark} 
The interpolation degree of a finite set is strongly related with the uniqueness degree in the reconstruction problem, see~\cite{DECASTRO2012336,GARCIA2021251}.
Let $K\subset \R^n$ be a compact set, and denote by $\mathcal{S}(K)$ the set of all signed Borel measures on $K$. Let $X\subset K$ be a finite set.
The uniqueness degree of $X$, denoted as $d(X)$, is defined to be the smallest integer $d$ such that for every measure $\mu$ supported on $X$, the  following \textit{superresolution} problem:
\begin{equation}\label{eq:supperr}
\begin{array}{ll}
\min_{\nu \in \mathcal{S}(K)} &\|\nu\|_{TV} 
\\
\mathrm{s.t.} \enspace & \<\nu, x^{\alpha} >=\<\mu,x^{\alpha}>,\enspace \forall \alpha\in \NN^n, \alpha_1+\cdots+\alpha_n\leq d
\end{array}
\end{equation}
has a unique solution. 
Here $\|\nu\|_{TV}$ denotes the total variation norm of $\nu$.
In \cite[Theorem 1.2]{GARCIA2021251}, it was demonstrated that $d(X)\leq \max\left(2g(X),i(X)\right)$, where $g(X)$  represents the maximum degree of a minimal generator of $\cI_{\R}(X)$ and $i(X)$ is the interpolation degree of $X$. Our approach for proving~\Cref{claim1}, as detailed in \Cref{thm:dimSI} and \Cref{seca:thmdim}, shares a strong similarity with the method used in \cite{GARCIA2021251}. Both approaches involve embedding into projective space to facilitate arguments with graded rings and to estimate degree bounds using the Hilbert function. However, a crucial difference lies in the nature of the embedding: we embed $V_{\CC}(g_1,\ldots,g_n)\subset \CC^n$ into the complex projective space, whereas \cite{GARCIA2021251} embeds the set $X\subset \R^n$  into the real projective space. This subtle yet significant distinction, which sets our approach apart from existing ones, echoes the discussion we previously had in \Cref{rem:c923M}.  
\end{remark}
To conclude this section, we also present an effective version of the results from \cite{LasserreLaurentRostalski08} and \cite{Laurent2012}.  
\begin{theorem}\label{thm:realradical}
Under~\Cref{ass:ginfty},  there exists some \begin{equation}\label{b:rankstablise}
t\leq  \lceil\max\left( (\fn+d_g)/2, \fn+d_h/2,\fn+1\right) \rceil,
\end{equation} such that
\begin{enumerate}
\item [(i)]
The set $\cK_{2t}$
is equal to the convex hull of the set $\{[v]_{2t}:v\in \cF\}$,
\item [(ii)] Let $\Lambda$ be a generic element of $\cK_{2t}$. Then 
$
\ker M_{t}(\Lambda)
$
generates the real radical ideal $\cI_{\R}({\cF})$.
\end{enumerate}
\end{theorem}
\begin{proof}
Let $
d=\lceil\max\left( (\fn+d_g)/2, \fn+d_h/2,\fn+1\right) \rceil.
$
From~\Cref{thm:flatorder},  for every $\Lambda \in \cK_{2d}$, we have
$\rank(M_{s}(\Lambda))=\rank(M_{t+s}(\Lambda))$, with $s=\theta$ and $t=d-\theta$ satisfying~\eqref{eq:scondition}. It follows from~\Cref{thm:L04} that $\Lambda_{2d}$ belongs to the convex hull of $\{[v]_{2d}:v\in \cF\}$.  Therefore, the set $\cK_{2d}$ is contained in the convex hull of the set $\{[v]_{2d}: v\in \cF\}$.  Thus (i) is proved.  To show (ii), let $\Lambda$ be a generic element of $\cK_{2d}$. Then $\rank(M_d(\Lambda))=|\cF|$ because the interpolation degree of $\cF$ is less than $d$. It remains to apply~\cite[Proposition 3.6]{LasserreLaurentRostalski08}. 
\end{proof}

\section{Application to SOS over gradient ideal}\label{sec:gsos}
In this section, we apply our main results to the case of SOS relaxation over the gradient ideal as considered in~\cite{NieDemmelSturmfels06}.

Let $F\in \R[x_1,\ldots,x_n]$. Suppose that the optimal value $F_{*}$ of the unconstrained optimization problem
\begin{equation}\label{a:minf}
\inf_{x\in \R^n} F(x)
\end{equation}
is finite and achieved. Then the search of optimal solution of~\eqref{a:minf} can be restricted on the  set of critical points, i.e. where the gradient of $F$ vanishes. Instead of solving directly~\eqref{a:minf},~\cite{NieDemmelSturmfels06} proposed to consider the problem constrained on the critical set:
\begin{equation}\label{eq:gradient}
 \begin{array}{lll}
   &\displaystyle\min_{x\in \R^n}~ & F(x) \\
 ~~~~~~~~&\mathrm{s.t.} &  \displaystyle \frac{\partial F}{\partial x_1}(x)=\cdots= \frac{\partial F}{\partial x_n}(x)=0.\\
 \end{array}
\end{equation}
When applying the SOS relaxation~\eqref{eq:13} to   problem~\eqref{eq:gradient}, we obtain the following  program:
\begin{equation}\label{eq:gradientsosd}
 \begin{array}{lll}
 \displaystyle F_d:=&\displaystyle\sup_{c\in \R} & c  \\
 ~~~~~~~~&\mathrm{s.t.} &  F-c \in \I{\frac{\partial F}{\partial x_1}, \ldots, \frac{\partial F}{\partial x_n}}_{2d}+ {\Sigma_{2d}} 
 .
 \end{array}
\end{equation}
{
Here, $\Sigma_{2d}:=\Sigma[x_1,\ldots,x_n]\cap \R[x_1,\ldots,x_n]_{2d}$ is the set of SOS polynomials of degree at most $2d$.}
This corresponds to the gradient type SOS relaxation for unconstrained problem~\eqref{a:minf} proposed in~\cite{NieDemmelSturmfels06}.  
 The set of critical points is  finite if 
 there does not exist $z\in \CC^n \backslash \{{\bf{0}}\}$ such that
 \begin{equation}\label{eq:Finfty}
 \frac{\partial F^{\infty}}{\partial x_1}(z)=\cdots= \frac{\partial F^{\infty}}{\partial x_n}(z)=0,
 \end{equation}
where $F^{\infty}$ is $F$'s homogeneous part of the highest degree (see~\Cref{sec:pre} for more precise definition).

As a direct corollary of~\Cref{thm:main}, we obtain the following degree bound for the exactness of~\eqref{eq:gradientsosd}.
\begin{corollary}\label{coro:gsos}
Let $F\in\R[x_1,\ldots,x_n]$ be a polynomial function such that 
\begin{equation}\label{a:Fstarfi}F_*:=\inf_{x\in \R^n} F(x)>-\infty\end{equation} and
\begin{equation}\label{a:resFinf}
\Res\left( \frac{\partial F^{\infty}}{\partial x_1}, \ldots, \frac{\partial F^{\infty}}{\partial x_n}\right)\neq 0,
\end{equation}
then 
$F_d=F_*$ for some
\begin{equation}\label{a:dboundsosg}
d\leq  \lceil \max\left(\deg(F)/2, n(\deg(F)-2) \right)\rceil. 
\end{equation}
If, in addition, $F''(x^*)$ is positive definite  at {all} global minimum point $x^*$, then 
\begin{equation}\label{a:Fsos}
F-F_*  \in \I{\frac{\partial F}{\partial x_1}, \ldots, \frac{\partial F}{\partial x_n}}_{2d}+ {\Sigma_{2d}} 
\end{equation}
for some $d$ bounded as in~\eqref{a:dboundsosg}.
\end{corollary}
\begin{proof}
By~\cite[Theorem 2.5(e)]{Nie13Jacobian}, the conditions~\eqref{a:Fstarfi} and~\eqref{a:resFinf} guarantee that the minimum of $F$ is achievable. Hence, the minimum value of~\eqref{eq:gradient} is also equal to $F_*$.  Next we  apply~\Cref{thm:main} and~\Cref{cor:main} to the optimization problem~\eqref{eq:gradient}. For this, let:
$$
g_1:=\frac{\partial F}{\partial x_1},\cdots,g_n:=\frac{\partial F}{\partial x_n}.
$$
Condition~\eqref{a:resFinf} implies that $\frac{\partial F^{\infty}}{\partial x_i}\neq 0$   for each $i\in [n]$. Consequently, 
$g^{\infty}_i=\frac{\partial F^{\infty}}{\partial x_i}$ for each $i\in [n]$ and thus~\Cref{ass:ginfty} is satisfied for~\eqref{eq:gradient}. Moreover, $F''(x^*)$ is positive definite whenever $x^*$ is a global minimum point implies that 
 the rank of 
 $
 F''(x^*)$ is $n$ for any global minimum point $x^*$. Since
 $F''(x^*)=\left[\begin{array}{lll}\nabla g_1(x^*) & \cdots & \nabla g_n (x^*)\end{array} \right ]
 $, the rank condition~\eqref{a:aradf} is satisfied when $F''(x^*)$ is positive definite. Finally note that 
 $$
 \sum_{i=1}^n \deg(g_i)-n=\sum_{i=1}^n (\deg(F)-1)-n=n\left(\deg(F)-2\right).
 $$
\end{proof}
\begin{remark}
For ease of comparison, we recall two known theorems closely related with~\Cref{coro:gsos}.
\begin{theorem}[\cite{Nie13Jacobian}]
Let $F\in \R[x_1,\ldots,x_n]$ be a polynomial function such that~\eqref{a:Fstarfi} and~\eqref{a:resFinf} hold. 
Then the minimum of $F$ is achievable and 
$F_d=F_*$ if $d$ is sufficiently large.
\end{theorem}
\begin{theorem}[\cite{Marshall_2009}]
Given positive integers $n$ and $\delta$, there exists a positive integer $\ell$ such that, for any $F\in \R[x_1,\ldots,x_n]$ of degree $\leq \delta$, if $F$ achieves a minimum value $F_*$ on $\R^n$, and the matrix $F''(x^*)$ is positive definite for each minimum point $x^*$ of $F$ on $\R^n$,   then~\eqref{a:Fsos} holds 
for some $d\leq \ell$. 
\end{theorem}
The main difference between~\Cref{coro:gsos} and the two above theorems is that an explicit expression on the degree bound is now available under conditions~\eqref{a:Fstarfi} and~\eqref{a:resFinf}. As noted in~\cite{Nie13Jacobian},~\eqref{a:resFinf} is a condition which holds generically. In particular, a sufficiently small random perturbation of the function $F$ will make~\eqref{a:resFinf} hold with probability one. The degree bound~\eqref{a:dboundsosg} is expected to be suboptimal as we merely applied~\Cref{thm:main} without taking into account of the special relation between the objective function and the constraints in the optimization problem~\eqref{eq:gradient}.  
\end{remark}

{
\begin{remark}In a more recent work,~\cite{Magron24Gradient} presents an SOS representation with rational coefficients for nonnegative polynomial over its gradient ideal, under the assumption that the gradient ideal $\langle \frac{\partial F}{\partial x_1}, \ldots, \frac{\partial F}{\partial x_n} \rangle$ is zero-dimensional and radical.  
This SOS representation with rational coefficients relies on the explicit construction of a  Gr\"obner basis of special form of the gradient ideal $\langle \frac{\partial F}{\partial x_1}, \ldots, \frac{\partial F}{\partial x_n}\rangle$. 
The  degree of the sum of squares term in their construction is bounded above by $\deg(F) \delta $, where $\delta$ is the number of points in the complex gradient variety $\{z\in \CC^n: \nabla F(z)=0\}$. Note that under the radical ideal assumption and the same conditions as in~\Cref{coro:gsos},  B\'ezout's theorem tells us  that $\delta=(\deg(F)-1)^n$.
\end{remark}
}

\section{Discussions on the assumption}\label{sec:cwp}

In this section we make some discussions on~\Cref{ass:ginfty}, which is critical for establishing the effective degree bound. We shall restrict our discussion to the case when  there are exactly $n$ equality constraints and no inequality constraints so that the degree bound~\eqref{a:degd}  reduces to:
\begin{equation}\label{a:deg22} \left\lceil \max\left( \deg(f)/2, \sum_{i=1}^n \deg(g_i)-n\right)\right\rceil. \end{equation}
\subsection{Examples}
We provide three explicit examples satisfying~\Cref{ass:ginfty} and discuss the relation with other assumptions in the literature.  
\begin{example}\label{ex-d=2}
Let $a, b \in \R$ such that $a\neq 0$, $b\neq 0$ and $a+b\neq 1$.
Let $n=2$ and consider
\begin{equation}\label{a:g1g2}
g_1(x_1,x_2)=x_2(a(x_2-1)-x_1(b-1)),\enspace g_2(x_1,x_2)=x_1(x_2(a-1)-b(x_1-1)).
\end{equation}
It can be checked easily that $(g_1,g_2)$ satisfies~\Cref{ass:ginfty}. The common zero locus of~\eqref{a:g1g2} are 
$$
V_{\CC}(g_1,g_2)=\left\{ (0;0), (1;0), (0;1), (a;b)\right\}.
$$
When $(a,b)=(1,1)$, we recover the binary problem of dimension $n=2$. When $(a,b)\neq (1,1)$, we obtain examples satisfying~\Cref{ass:ginfty} which can not be made equivalent to the binary example through affine transformation. Moreover, it is easy to see that
\begin{align*}
g_1(x)=x_2(a(x_2-1)-x_1(b-1)),\enspace g_2(x)=x_1(x_2(a-1)-b(x_1-1)), \\
g_3(x)=x_4(c(x_4-1)-x_3(d-1)),\enspace g_4(x)=x_3(x_4(c-1)-d(x_3-1))
\end{align*}
 is a 4-dimensional example which satisfies~\Cref{ass:ginfty} if $a\neq 0$, $b\neq 0$, $c\neq 0$, $d\neq 0$, $a+b\neq 1$ and $c+d\neq 1$.  Again, the binary  problem only corresponds to the isolated case when $a=b=c=d=1$. Higher dimensional examples which satisfies~\Cref{ass:ginfty} but are not equivalent to the binary problem by affine transformation can be constructed in a similar way.
\end{example}
Note  that the ideal  $\<g_1,g_2>$ in the above examples is  radical. Below is a 2-dimensional example satisfying~\Cref{ass:ginfty} which generates non radical ideal.
\begin{example}\label{ex:notradical}
Let $a, b \in \R$ such that  $a\neq 0$ and $a+b\neq 1$. 
Let $n=2$ and consider
\begin{equation}\label{a:g1g23}
g_1(x_1,x_2)=x_2(a(x_2-1)-x_1(b-1)),\enspace g_2(x_1,x_2)=x_1(x_1+x_2-1).
\end{equation}
Then~\Cref{ass:ginfty} is satisfied. The common zero locus of~\eqref{a:g1g23} are:
$$
V_{\CC}(g_1,g_2)=\left\{ (0;0), (1;0), (0;1) \right\}.
$$
The ideal $\<g_1,g_2>$ is not  radical. 
\end{example}

We know from~\Cref{thm:dimSI} that~\Cref{ass:ginfty} implies the H-basis property. 
If $\{g_1,\ldots,g_n\}$ is a Gr\"{o}bner basis of the ideal $\<g_1,\ldots,g_n>$ { with respect to an ordering that respects the degree}, then $\{g_1,\ldots, g_n\}$ also forms an H-basis. However,~\Cref{ass:ginfty} is fundamentally different from the Gr\"{o}bner basis assumption.
Below is a two-dimensional example  which satisfies~\Cref{ass:ginfty} without being a Gr\"{o}bner basis.
\begin{example}\label{ex:cgrob}
Consider
\begin{align*}
g_1(x_1,x_2)=(x_1-1)^3+(x_2-1)^3,\enspace g_2(x_1,x_2)=(x_1-1)^4 (x_2-1)^4.
\end{align*}
\Cref{ass:ginfty} is  satisfied. 
We have
$$
\psi(x_1,x_2)=((x_2-1)^7-(x_1-1)^3 (x_2-1)^4) g_1(x_1,x_2)+ (x_1-1)^2 g_2(x_1,x_2)=(x_2-1)^{10},
$$
and 
$$
\phi(x_1,x_2)=((x_1-1)^7-(x_1-1)^4 (x_2-1)^3) g_1(x_1,x_2)+(x_2-1)^2 g_2(x_1,x_2)=(x_1-1)^{10}.
$$
Let $\LT(p)$ denote the leading term of $p$ with respect to a certain monomial ordering. 
Then,  $\LT(\psi)=x_2^{10}$, $\LT(\phi)=x_1^{10}$ and it is impossible that both of them are contained in the monomial ideal generated by $\LT(g_1)$ and $\LT(g_2)$. In other words, $\{g_1, g_2\}$ is not a Gr\"{o}bner basis  with respect to any monomial ordering. 

More generally, if $g_1,\ldots,g_n$ satisfies~\Cref{ass:ginfty}  and also forms a Gr\"{o}bner basis, it is necessary that for each $i\in [n]$, there is some $g_j$ such that $\LT(g_j)$ is a nonnegative power of $x_i$, by~\cite[Theorem 6, \S 5.3]{CoxLittleOshea15}.    In other words, if $\{g_1,\ldots,g_n\}$ satisfies~\Cref{ass:ginfty} and there is some $i\in [n]$ such that for all $j\in [n]$, $\LT(g_j)$ is not a nonnegative power of $x_i$, then $\{g_1,\ldots,g_n\}$ can not be a Gr\"{o}bner basis. 
\end{example}

\subsection{Genericity of the assumption}\label{a:gofa}
\Cref{ass:ginfty} is satisfied \textit{generically} in the following sense. Let $d_1,\ldots,d_n$ be fixed integers. Consider the vector space $\R[x_1,\ldots,x_n]_{\leq d_1}\times \cdots \times \R[x_1,\ldots,x_n]_{\leq d_n}$ of 
 $n$ real polynomial functions of degrees bounded by $d_1,\ldots,d_n$.
\Cref{ass:ginfty} is satisfied if and only
$$
\Res\left(g^{\infty}_1,\ldots,g^{\infty}_n\right)\neq 0,
$$
which corresponds to the complement of a hypersurface in $\R[x_1,\ldots,x_n]_{\leq d_1}\times \cdots \times \R[x_1,\ldots,x_n]_{\leq d_n} $. In particular, the coefficients defining $g_1,\ldots,g_n$ satisfying~\Cref{ass:ginfty} forms a Zariski open set, in the space of coefficients of $n$ real polynomials of fixed degrees.

Note that the optimization problem~\eqref{eq:1} seeks to minimize the function $f$ on a finite subset $W\subset \R^n$. Given the finite subset  $W\subset \R^n$, 
the choice of a set of $n$ polynomials $\{g_1,\ldots,g_n\}$  that defines the solution set $W$ should be viewed as an auxiliary data associated to the original problem~\eqref{eq:1}, though the sum of squares relaxation will depend on this choice. Given an arbitrary finite subset $W\subset \RR^n$, it is  possible to find $\{g_1,\ldots,g_n\}\subset \RR[x_1,\ldots,x_n]$ such that  $W=\{x\in \RR^n: g_1(x)=\cdots=g_n(x)=0\}$ and~\Cref{ass:ginfty} holds, see~\cite{Eisenbud1973}. 
With such a choice of defining polynomials,  an effective degree bound for the tightness of Lasserre's SOS relaxation with respect to arbitrary polynomial objective function $f$ is given by~\eqref{a:degd}. 

Given an arbitrary finite set $W\subset \R^n$, 
let $\{g_1,\ldots,g_n\}$ be polynomials such that $W=\{x\in \RR^n: g_1(x)=\cdots=g_n(x)=0\}$ and~\Cref{ass:ginfty} holds. It often occurs that some of the points in $W$ are singular. 
    We have already seen such an example in~\Cref{ex:notradical}. More generally, consider  the example when $W\subset \R^n$ is a subset of $|W|=2^n-1$ points. Let $\{g_1,\ldots,g_n\}\in \R[x_1,\ldots,x_n]$ be quadratic functions such that $W=\{x\in \R^n: g_1(x)=\cdots=g_n(x)=0\}$ and~\Cref{ass:ginfty} holds. 
By B\'ezout's theorem, at least one point in $W$ has multiplicity, and so $\<g_1,\ldots,g_n>$ can not be radical ideal.  Recall that
  singularity is usually an undesirable property in optimization both from theoretical and practical aspects. Interestingly,  appropriately adding  singularity on points in $W$ greatly simplifies the effective degree bound problem.

 \subsection{Comparison with grid case}\label{subsec:cwp}
 Consider the grid case when  the polynomials $g_1,\ldots,g_n$ are given by:
  \begin{equation}\label{a:grid}
  g_i(x)=\prod_{j=1}^{d_i}(x_i-a^i_j),\enspace \forall i =1,\ldots,n.
  \end{equation}  
  This grid case satisfies the technical assumption (\Cref{ass:ginfty}) and the degree bound~\eqref{a:deg22} applies.
   This degree bound was previously obtained in~\cite{Laurent07} for the grid case defined in~\eqref{a:grid}.
  
 Fix a positive integer $r$. Let us study the difference between  the class of problems covered by grid case and the class covered by~\Cref{ass:ginfty}. Note that the degree bound for the exactness of Lasserre's SOS relaxation is determined by the $\R$-vector space spanned by $\{g_1,\ldots,g_n\}$. Therefore, if suffices to consider the space of $n$-dimensional linear subspace of the vector space $\R[x_1,\ldots,x_n]_{\leq r}$, which is the Grassmannian $\mathbf{Gr}(n,\binom{n+r}{r} )$.
 Furthermore, two subspaces $\mathrm{span}\{g_1,\ldots,g_n\}$ and $\mathrm{span}\{\tilde g_1,\ldots,\tilde g_n\}$ are said to be equivalent if there is an invertible affine transformation $L:\R^n\rightarrow \R^n$ such that
  $$
  \tilde g_i(x)=g_i(L(x)),\enspace \forall i\in [n],\enspace x\in \R^n. 
  $$
  Denote by $\sim$ this equivalence relation. 
  The dimension of $\mathbf{Gr}(n,\binom{n+r}{r} )/\sim$ is at least 
  \begin{equation}\label{a:assesdfe}
\left(  \binom{n+r}{r}-n\right) n-(n^2+n),
  \end{equation}
  where the first component is the dimension of the Grassmannian and the second component is the dimension of invertible affine transformation over $\R^n$.  Since~\Cref{ass:ginfty} is satisfied on an open set, the dimension of the equivalent class of problems satisfying~\Cref{ass:ginfty} is the same as~\eqref{a:assesdfe}.
  For the grid case defined in~\eqref{a:grid},  it is always possible to let $a^i_1=1$ and $a^i_2=-1$ for each $i\in [n]$ by invertible affine transformation. Hence the dimension of the equivalent classes represented by the grid case~\eqref{a:grid} is at most
  \begin{equation}\label{a:wdgridcdim}
  (r-2)n. 
  \end{equation}
  When $n=2$ and $r=2$, the dimension formula~\eqref{a:assesdfe}  returns 2 and the dimension formula~\eqref{a:wdgridcdim} return 0. This is compatible with what we found previously in~\Cref{ex-d=2}. When $r$ is fixed and $n$ increases, the dimension of problems satisfying~\Cref{ass:ginfty} grows as $O\left(n^r\right)$, while the dimension of problems corresponding to the grid case grows as $O(rn)$, from which we see the magnitude of difference between these two classes.

 Finally, recall that more refined degree bound was obtained in~\cite{Fawzi,SakaueTakeda17} for the binary optimization problem. However, their approach relies on the finite abelian group structure of the underlying feasible set, while the degree bound in this paper is established under a condition which is satisfied generically.

\subsection{Correlation of points}
The effective degree bound established in~\eqref{a:degd} is essentially determined by the number $\fn$ defined in~\eqref{a:fn}, which increases with the sum of the degrees of the generating functions $\{g_1,\ldots,g_n\}$. This number has a close relation with the correlation of points in $W$.

   Suppose that we search for $\{g_1,g_2,g_3\}\in \RR[x_1,x_2,x_3]_{\leq 2}$   which vanishes at 8 given points in $\R^3$ and also satisfies~\Cref{ass:ginfty}. The number of coefficients defining a degree 2 polynomial of 3 variables is 10 and one vanishing point determines one linear equation on the  coefficients. Hence, if the number of linearly independent equations on the coefficients is $8$, which occurs if these 8 points are in general positions, it is impossible to have 3 linearly independent solutions. 
   Saying differently, to find  $\{g_1,g_2,g_3\}\in \RR[x_1,x_2,x_3]_{\leq 2}$   which vanishes at 8 given points in $\R^3$ and also satisfies~\Cref{ass:ginfty}, the 8 given points must be correlated and the maximal number of free points (i.e., points in general positions) is 7. 
   
   More generally,  the more free points in $W\subset \R^n$ there are, the larger the degree bound $\fn$ is.  To find $\{g_1,\ldots,g_n\}\subset \RR[x_1,\ldots,x_n]_{\leq r}$ which vanishes at $W$ and satisfies~\Cref{ass:ginfty}, the  number of free points in $W$ must be less than
   $$
   \binom{n+r}{n}-n.
   $$
   Let $W\subset \R^n$ be a subset of  $4^n$ points. If the  points in $W$ are in general positions, then $r$ is at least $n$ because $\binom{2n}{n}-n\leq 4^n$. In turn, the degree bound $\fn$ is at least $n^2-n$. In contrast, if the $4^n$ points are strongly correlated so that there are $\{g_1,\ldots,g_n\}\subset \RR[x_1,\ldots,x_n]_{\leq 4}$ which vanishes at $W$ and satisfies~\Cref{ass:ginfty}, then the degree bound $\fn$ can be reduced to $3n$. Hence, between points with correlation and points of random positions, the degree bound can range from $O(n)$ to $O(n^2)$.

\section{Conclusion and perspectives}\label{sec:perspec}

In this paper, we propose to study the effective degree bound under the assumption of nonexistence of solution at infinity. This geometric condition is characterized by the non-vanishing of the resultant of the homogeneous part of the highest degree of $g_1, \ldots,g_n$. Under this assumption which holds on a Zariski open set in the space of coefficients defining $g_1,\ldots,g_n$ of fixed degrees, we obtain an effective degree bound for the representation of $f-f_*$ in the quadratic module $\<g_1,\ldots,g_k>+\qm{h_0,h_1,\ldots,h_m}$.  This degree bound is explicitly computed from the degrees of $g_1,\ldots,g_k,h_1,\ldots,h_m$. As a direct application, we establish the first explicit degree bound for SOS representation over the gradient ideal under a generic condition. Our approach also brings some new insights for the study of  complexity for polynomial optimization. Our result confirms that the correlation of points in the feasible region has a direct impact on the degree bound, and suggests to reconsider the role played by  multiplicity in optimization. 

The geometric condition guarantees that $\{\bar g_1,\ldots,\bar g_n\}$ forms a regular sequence, from which we obtain the effective bound from the explicit form of  the Hilbert series of the quotient ring. However,~\Cref{ass:ginfty} is stronger than the regular sequence property of $\{\bar g_1,\ldots,\bar g_n\}$, see~\Cref{seca:thmdim}. Indeed, $\{\bar g_1,\ldots,\bar g_n\}$ is still a regular sequence if $\{g^{\infty}_1,\ldots,g^{\infty}_n\}$ has finitely many complex common zero locus. However, the H-basis property may fail if we only assume that  $\{\bar g_1,\ldots,\bar g_n\}$ is a regular sequence. It would be interesting to investigate the degree bound under the regular sequence condition only. 

In this paper we considered the case when 
 the number of equality constraints $k$ is at least $n$. When $k<n$, one may consider the equivalent polynomial optimization of restricted on the KKT ideal~\cite{DEMMEL2007189,Nie13Jacobian}:
\begin{equation}\label{a:minfconkkt}
\begin{array}{ll}
 \displaystyle \min_{\substack{x\in \R^n, \lambda \in \R^{k+m}}}~ & f(x) \\
 ~~\mathrm{s.t.} &  g_1(x)=\cdots=g_k(x)=0\\
&  \nabla f(x)= \sum_{i=1}^k \lambda_i \nabla g_i (x)+\sum_{j=1}^m \lambda_{k+j} h_j\\
& \lambda_{k+1}h_1(x)=\cdots=\lambda_{k+m}h_m(x)=0\\
& h_1(x)\geq 0,\ldots,h_m(x)\geq 0\\
& \lambda_{k+1}\geq 0,\ldots, \lambda_{k+m}\geq 0
 \end{array}
 \end{equation}
 Note that the number of variables in~\eqref{a:minfconkkt} is $n+m+k$ and the number of equality constraints  is also $n+m+k$. However, unlike the gradient ideal,~\Cref{ass:ginfty} is rarely satisfied for the KKT ideal~\eqref{a:minfconkkt}. This suggests to consider other conditions for establishing degree bounds for SOS representation over the KKT ideal. 
 
 Finally, a more challenging question would be to improve the degree bound given in~\Cref{coro:gsos} by exploiting the relation between the objective function and the defining equations. It is also interesting to study  problems with sparsity structures for obtaining refined effective degree bound.

\begin{appendices}


\section{Proof of~\Cref{l:nonsingular}}
Hereinafter, $\bf{i}$ denotes the imaginary unit.

\begin{proof}[Proof of~\Cref{l:nonsingular}.]
 Assume w.l.o.g. that $ v_1=\cdots= v_{n-1}=0$ and $ v_n=\bf{i}$ and the maximal ideal is
$\fm=\<x_1,\ldots,x_{n-1}, x_n^2+1>$. (The case when $v$ is  real  can be proved in a similar way by considering the maximal ideal $\fm=\<x_1,\ldots,x_{n-1}, x_n>$.) 

 Let $p(x):=x_n^2+1$.   Since $\left[\begin{array}{ccc}
    \nabla g_1(v) & \cdots & \nabla g_k(v)  
\end{array}\right] $ has rank $n$, there exist 
$(a_1,\ldots,a_k)\in \RR^k$ and $(b_1,\ldots,b_k)\in \RR^k$ such that
$$
\left[\begin{array}{ccc}
    \nabla g_1(v) & \cdots & \nabla g_k(v)  
\end{array}\right] \left[\begin{array}{c}
     a_1+{\bf{i}} b_1 \\
     \vdots\\
     a_k+{\bf{i}} b_k
\end{array}\right]= \nabla p(v).
$$
Let  $h\in \< g_1,\ldots,g_k> $ be defined by:
$$
h(x):=\sum_{j=1}^k (a_j+b_j x_{n})g_j(x).
$$
It follows that $h-p\in \fm$ and 
$$
\nabla h(v)=\sum_{j=1}^k (a_j+{\bf{i}} b_j) \nabla g_j(v)=\nabla p(v). 
$$
Therefore, $h-p\in \fm^2$ and it follows that $p\in \fm^2+\<g_1,\ldots,g_k>$. In a similar way, we can show that $x_i \in \fm^2+\<g_1,\ldots,g_k>$ for each $i\in [n-1]$. Thus $\fm=\fm^2+\<g_1,\ldots,g_k>$ and it follows that  
$\fm=\fm^\ell+\<g_1,\ldots,g_k>$ for any $\ell\geq 2$.
\end{proof}

\section{Proof of~\Cref{l:squarerootlemma01}}
In this section prove~\Cref{l:squarerootlemma01}. We start by proving a general result about taking square root in the quotient  of a ring of polynomials $R[x_1,\ldots,x_n]$ where $R$ is an arbitrary ring. A \textit{unit} element $b\in R$ is an element such that $bw=1$ for some $w\in R$.
\begin{lemma}\label{l:squareroot}
Let $R$ be a ring and $R[x_1,\ldots,x_n]$ be the ring of polynomials of $n$ variables with coefficients in $R$. Denote by $\langle x_1,\ldots,x_n\rangle$ the ideal of $R[x_1,\ldots,x_n]$ generated by $x_1,\ldots,x_n$.
Let $p\in R[x_1,\ldots,x_n]$ such that there is 
a unit $b\in R$ satisfying $b^2=p(0,\ldots,0)$.
Then for any $\ell\geq 1$, there is $q\in R[x_1,\ldots,x_n]$  such that 
$
p=q^2
$ in $R[x_1,\ldots,x_n]/\langle x_1,\ldots,x_n\rangle ^\ell$.
\end{lemma}
\begin{proof}
Let $\{a_{\alpha}: \alpha \in \NN^n\}$ such that $$p(x)=\sum_{\alpha \in \NN^n} a_{\alpha} x^{\alpha}.$$
In the above, $x^\alpha$ is the monomial $x^{\alpha_1}_1\cdots x^{\alpha_n}_n$. We denote by $|\alpha|$ the degree of the monomial $x^\alpha$, i.e., $|\alpha|=\alpha_1+\cdots+\alpha_n$.
Let any $\ell\geq 1$. 
Let $\{b_{\alpha}: \alpha\in \NN^n, |\alpha|\leq \ell-1\}$ be defined by  the following recursive system:
\begin{align*}
 & b_{{\bf{0}}}=b,\enspace \\
 &
  b_{\alpha}=\frac{1}{2}b_{{\bf{0}}}^{-1}\left(a_{\alpha}-\sum_{\substack{(\beta_1, \beta_2)\in \NN^n\times \NN^n\\ \beta_1+\beta_2=\alpha\\ \beta_1,\beta_2 \neq {\bf{0}}} } b_{\beta_1}b_{\beta_2}\right),\enspace \alpha \in \NN^n,  \alpha \neq {\bf{0}}, |\alpha|\leq \ell-1.
\end{align*}
Furthermore, let $b_{\alpha}=0$ for all $\alpha$ such that $|\alpha|\geq \ell$.
It is easy to verify that 
$$
a_{\alpha}=\sum_{\substack{(\beta_1, \beta_2)\in \NN^n\times \NN^n\\ \beta_1+\beta_2=\alpha}} b_{\beta_1}b_{\beta_2},\enspace   \alpha\in \NN^n, |\alpha|\leq \ell-1.
$$
Let $q\in R[x_1,\ldots,x_n]$ be defined by $
q(x)=\sum_{\alpha \in \NN^n} b_{\alpha} x^{\alpha}.
$
Then 
\begin{align*}&
p(x)-q(x)^2\\&=\sum_{\substack{\alpha \in \NN^n\\ |\alpha| < \ell}  } \left(a_{\alpha}-\sum_{\substack{(\beta_1, \beta_2)\in \NN^n\times \NN^n\\ \beta_1+\beta_2=\alpha}} b_{\beta_1}b_{\beta_2} \right) x^{\alpha}+\sum_{\substack{\alpha\in \NN^n \\ |\alpha|\geq \ell}}   \left(a_{\alpha}-\sum_{\substack{(\beta_1, \beta_2)\in \NN^n\times \NN^n\\ \beta_1+\beta_2=\alpha}} b_{\beta_1}b_{\beta_2} \right) x^{\alpha}
\\
&=\sum_{\substack{\alpha\in \NN^n\\ |\alpha|\geq \ell}}   \left(a_{\alpha}-\sum_{\substack{(\beta_1, \beta_2)\in \NN^n\times \NN^n\\ \beta_1+\beta_2=\alpha}} b_{\beta_1}b_{\beta_2} \right) x^{\alpha}\in \langle x_1,\ldots,x_n\rangle ^\ell.
\end{align*}
\end{proof}

\begin{lemma}\label{l:sqC}
Let $p\in \R[x]$ be a one-variable polynomial such that $p({\bf{i}})\neq 0$. Then for any $\ell\geq  1$, there is a unit $q\in \R[x]/\langle x^2+1\rangle ^\ell $ such that $p=q^2$ in $\R[x]/\langle x^2+1\rangle ^\ell$.
\end{lemma}
\begin{proof}
Let any $\ell\geq 1$. By~\Cref{l:squareroot}, there is a polynomial $u\in \R[x]$ such that $u^2(x)=1-x$ in $\R[x]/\<x>^{\ell}$.  Since $u(0)=1$, we know that $1-u(x)\in \<x>$ and so $1-u(x^2+1)\in \<x^2+1>$.
It follows that
\begin{equation}\label{eq:1minus}
\left( \sum_{i=0}^{\ell-1} \left(1-u(x^2+1)\right)^{i}  \right) u(x^2+1) =1-\left(1-u(x^2+1)\right)^{\ell} \in 1+\<x^2+1>^\ell.
\end{equation}
Since 
$
u^2(x^2+1) =-x^2 
$ in $\R[x]/\<x^2+1>^{\ell}$, we have
\begin{equation*}
\begin{aligned}
&x^2 \left(\sum_{i=0}^{\ell-1} \left(1-u(x^2+1)\right)^{i}\right)^2   {\in} -u^2(x^2+1) \left(\sum_{i=0}^{\ell-1} \left(1-u(x^2+1)\right)^{i}\right)^2 +\<x^2+1>^\ell.
\end{aligned}
\end{equation*}
Using~\eqref{eq:1minus}, we obtain:
\begin{equation}\label{eq:oneplusx2}
1+x^2 \left(\sum_{i=0}^{\ell-1} \left(1-u(x^2+1)\right)^{i}\right)^2 \in \<x^2+1>^\ell
\end{equation}
Since $({\bf{i}} u(z))^2+1 \in z+\<z>^{\ell}$, there is an $\RR$-algebra homomorphism $\phi:\R[x]/\langle x^2+1\rangle ^\ell \rightarrow \CC[z]/\<z>^\ell$ such that: 
$$
\phi(x)={\bf{i}}u(z).
$$
In view of~\eqref{eq:oneplusx2},
there is an $\RR$-algebra homomorphism $\psi: \CC[z]/\<z>^\ell\rightarrow \R[x]/\langle x^2+1\rangle ^\ell $ such that:
$$
\psi(z)=x^2+1,\enspace \psi({\bf{i}})= x  \left(\sum_{i=0}^{\ell-1} \left(1-u(x^2+1)\right)^{i}\right).
$$
We also check that
$$
\psi\circ \phi(x)= x  \left( \sum_{i=0}^{\ell-1} \left(1-u(x^2+1)\right)^{i}  \right) u(x^2+1)\in x+ \<x^2+1>^{\ell}.
$$
Hence $\psi\circ \phi(p)=p$ in $\R[x]/\<x^2+1>^\ell$.
 Let $h:=\phi(p)\in \CC[z]/\<z>^\ell$. Since $p({\bf{i}})\neq 0$, we have $h(0)\neq 0$. Applying~\Cref{l:squareroot} with $R=\CC$, we deduce that there is $w\in \CC[z]$ such that $w^2=h$ in $\CC[z]/\<z>^\ell$. {Furthermore}, since $w(0)\neq 0$, a similar argument shows that $w$ is a unit in $\CC[z]/\<z>^\ell$. 
 Let $q=\psi(w)$. Then $q^2=\psi(w^2)=\psi(h)=\psi\circ \phi (p)=p$ in $\R[x]/\<x^2+1>^\ell$.  
\end{proof}

Now we are ready to show~\Cref{l:squarerootlemma01}.

\begin{proof}[Proof of~\Cref{l:squarerootlemma01}.]
 Consider the case {where} $v$ is real. Without loss of generality, we may assume that  $v=(0,\ldots,0)$ and $$\fm=\langle x_1,\ldots,x_n\rangle.$$ 
Since $p({0,\ldots,0})>0$, $\sqrt{p({0,\ldots,0})}$ is a unit in $\RR$, it suffices to apply~\Cref{l:squareroot} with $R=\RR$.

Consider the case where $v$ is not a real point. Without loss of generality, we may assume that $v=(0,\ldots,0,{\bf{i}})$ and  $$\fm=\langle x_1,\ldots,x_{n-1},x^2_n+1\rangle.$$ Let $\mathbb{B}=\R[x_n]/\langle x_n^2+1\rangle ^\ell$. 
Then $\RR[x_1,\ldots,x_n]/\fm^\ell$ is a quotient ring of $$\mathbb{B}[x_1,\ldots,x_{n-1}]/\langle x_1,\ldots,x_{n-1}\rangle ^\ell.$$ Let $\{a_{\alpha}: \alpha\in \NN^{n-1}\}\subset \RR[x_n]$ such that 
 $p(x_1,\ldots,x_n)=\sum_{\alpha\in \NN^{n-1}} a_{\alpha} x_1^{\alpha_1}\cdots x_{n-1}^{\alpha_{n-1}} $. Since $p(0,\ldots,0,{\bf{i}})\neq 0$,  we have $a_{0,\ldots,0}({\bf{i}})\neq 0$. By~\Cref{l:sqC}, there is $b\in \mathbb{B}$ such that $a_{0,\ldots,0}=b^2$ and $b$ is a unit in $\mathbb{B}$.  It remains to apply~\Cref{l:squareroot} with $R=\mathbb{B}$. 
 \end{proof}
 Finally,~\Cref{l:squarerootlemma-1} is a special case of~\Cref{l:squarerootlemma01}.
 
\begin{proof}[Proof of~\Cref{l:squarerootlemma-1}.]
We apply~\Cref{l:squarerootlemma01} with $p$ being the constant function $-1$.
\end{proof}

\section{Proof of~\Cref{thm:dimSI}}\label{seca:thmdim}
We complete the sketched proof provided in~\cite{MollerHbasis} for
\Cref{thm:dimSI}. For references
of materials in this section, see~\cite{AtiyahMac,Eisenbudbook,CoxLittleOshea15}.

\subsection{Preliminaries}  Recall that $\bS=\R[x_0,x_1,\ldots,x_n]$. 
Let $\bI$ be an ideal of $\bS$. 
The associated primes of $\bI$ are the prime ideals which occur as radicals of annihilators of nonzero elements of $\bS/\bI$.
The minimal primes of  $\bI$  are  the minimal elements in the set of all prime ideals which contains $\bI$.   

The (Krull) \textit{dimension} of the quotient ring $\bS/\bI$, denoted by $\dim \left(\bS/\bI \right)$, is the supremum of the lengths of chains of prime ideals in $\bS$ containing $\bI$. This dimension coincides with the dimension of the  complex variety in $\CC^{n+1}$ associated with the ideal $\bI$.

The \textit{codimension} of a prime ideal $I$, is the supremum of the lengths of chains of primes descending from $I$. The \textit{codimension} of $\bI$,  is the minimum of the codimensions of the prime ideals containing $I$, and hence is the minimum of the codimensions of the minimal primes of $I$.
We have $\codim (\bI)=n+1-\dim\left(\bS/\bI\right)$.
\begin{theorem}[Unmixedness Theorem]\label{thm:unmixed}\cite[Corollary 18.14]{Eisenbudbook}
Let $h_1,\ldots,h_r\in \bS$.
If $\bI:=\langle h_1,\ldots,h_r\rangle$ is an ideal of  $\bS$  such that $\codim(\bI)=r$, then all the minimal primes of $\bI$ have codimension $r$ and every associated prime of $\bI$ is minimal over $I$.
\end{theorem}
 A polynomial $p\in \bS$ is a \textit{zero-divisor} of $\bS/\bI$ if  there is  $q\notin \bI$ such that $pq\in \bI$.
 A sequence of elements $h_1,\ldots,h_r\in \bS$ is a \textit{regular sequence} on $\bS$ if
\begin{itemize}
    \item[1.] $\langle h_1,\ldots,h_r\rangle \neq \bS$;
    \item[2.] for each $i=1,\ldots,r$,  $h_{i}$ is  not a zero-divisor of $\bS/\langle h_1,\ldots,h_{i-1}\rangle$. 
\end{itemize}
Let $\bI$ be a homogeneous ideal of $\bS$. 
 The \textit{Hilbert series} of $\bS/\bI$ is the formal power series in one variable $t$ given by 
 $$
 H_{\bS/\bI}(t)=\sum_{d=0}^{\infty} \left(\dim_{\R} \bS_d/\bI_d \right) t^d.
 $$
\begin{theorem}\cite[Exercise 7, P.137]{Froberg98}\label{thm:Hilbertseries}
Let $h_1,\ldots,h_r \in \bS$ be $r$ homogeneous polynomials of $n+1$ variables. Let $\bI$ be the homogeneous ideal of $\bS$ generated by $h_1,\ldots,h_r$. 
If $\{h_1,\ldots,h_r\}$ is a regular sequence on $\bS$, then  
$$
H_{\bS/\bI}(t)=\frac{ \displaystyle\prod_{i=1}^r \left(  \displaystyle \sum_{j=1}^{\deg({h_i})} t^{j-1} \right)}{\left(1-t\right)^{n+1-r}}.
$$
\end{theorem}

\begin{theorem}\label{thm:completeintersection} 
Let $h_1,\ldots,h_r \in \bS$ be $r$ homogeneous polynomials of $n+1$ variables. Let $\bI:=\langle h_1,\ldots,h_r\rangle$ be the ideal of~$\bS$ generated by $h_1,\ldots,h_r$. 
If $\codim \left(\bI\right)=r$,  then $h_1,\ldots, h_r$ is a regular sequence on $\bS$.
\end{theorem}
\Cref{thm:completeintersection} is a standard result in algebraic geometry, though we did not find a reference which contains exactly the same statement. For the sake of completeness we include   a proof of this theorem. 
 
\begin{proof}[Proof of~\Cref{thm:completeintersection}.]
For each $t\in [r]$, denote by $\bI(t)$ the ideal of $\bS$ generated by $h_1,\ldots,h_t$.  Since $\bI(t+1)=\bI({t})+\langle h_{t+1}\rangle$ for each $t\in [r-1]$, we have
$$
\codim (\bI(t+1))\leq \codim (\bI(t))+1,\enspace \forall t\in [r-1]. 
$$
Since $\codim(\bI(1))=1$ and $\codim(\bI(r))=r$, we have $\codim(\bI(t))=t$ for each $t\in [r]$. 
Suppose that there is some $t\in [r-1]$ such that  $h_{t+1}$ is a zero-divisor of $\bS/\bI(t)$. Then $h_{t+1}$ is in an associated prime $p$ of $\bI(t)$, by~\cite[Proposition 4.7]{AtiyahMac}. 
Thus $\bI(t+1)=\bI(t)+\langle h_{t+1}\rangle \subset p$. It follows that
$\codim \left(\bI(t+1)\right)\leq \codim (p)$.  Since  $\bI(t)=\langle h_1,\ldots,h_t\rangle$ is an ideal of codimension $t$ generated by $t$ elements, by~\Cref{thm:unmixed}, the associated prime $p$ of $\bI(t)$ is  a minimal prime  of $\bI(t)$ and
$\codim(p)=\codim (\bI(t))= t$.  
This contradicts with $\codim(\bI(t+1))=t+1$. Therefore, $h_1,\ldots,h_r$ is a regular sequence.  
\end{proof}

\subsection{Proof of~\Cref{thm:dimSI}}
An immediate consequence of~\Cref{ass:ginfty} is the following key lemma.
\begin{lemma}
Under~\Cref{ass:ginfty}, $\{\bar g_1,\ldots,\bar g_n, x_0\}$ is a regular sequence on $\bS$. 
\end{lemma}
\begin{proof}
The complex variety in $\CC^{n+1}$ defined by 
$$
\{z\in \CC^{n+1}: \bar g_1(z)=\cdots =\bar g_n(z)=0, z_0=0\}
$$
is the origin point $\{{\bf{0}}\}\subset \CC^{n+1}$. Hence the codimension of the ideal generated by $\{\bar g_1,\ldots, \bar g_n, x_0\}$ is $n+1$. It remains to apply~\Cref{thm:completeintersection}.
\end{proof}

Hereinafter, we denote by $\bI$ the homogeneous ideal of $\bS$ generated by the homogeneous polynomials $\bar g_1,\ldots ,\bar g_n$. 
Simply using the definition of regular sequence,
we immediately have the following results. 
\begin{corollary}\label{coro:completeintersection}
Under~\Cref{ass:ginfty}, $\{\bar g_1,\ldots,\bar g_n\}$ is a regular sequence on $\bS$. 
\end{corollary}
\begin{corollary}\label{coro:hbasis}
Under~\Cref{ass:ginfty}, if $p\in \bS$ satisfies $x_0 p \in \bI$, then $p\in \bI$. 
\end{corollary}
\begin{proof}[Proof of~\Cref{thm:dimSI}.]
Let $h\in \<g_1,\ldots,g_n>$ and $\lambda_1,\ldots,\lambda_n\in \RR[x_1,\ldots,x_n]$ such that
$$
h=\sum_{i=1}^n \lambda_i g_i.
$$
Since \begin{align*}
\ol{h}\left(x_0,\ldots,x_n\right)&=x_0^{\deg(h)} h\left(\frac{x_1}{x_0},\ldots, \frac{x_n}{x_0}\right)\\&=x_0^{\deg(h)}\sum_{i=1}^n \lambda_i\left(\frac{x_1}{x_0},\ldots,\frac{x_n}{x_0}\right)g_i\left(\frac{x_1}{x_0},\ldots,\frac{x_n}{x_0}\right),
\end{align*}
  there is an integer $\ell \geq 0$ such that
$
x_0^{\ell} \ol{h} \in \bI.
$
By~\Cref{coro:hbasis}, $\ol{h}\in \bI$. Consequently, there are homogeneous polynomials $\bar \lambda_1,\ldots,\bar\lambda_n\in \bS$ such that 
\begin{equation}\label{a:barh}
\ol{h}=\sum_{i=1}^n \bar \lambda_i \ol{g}_i.
\end{equation}
Since $\ol{h}$ is homogeneous, the $\bar \lambda_1,\ldots,\bar\lambda_n$ in~\eqref{a:barh} can be selected such that each $\ol{\lambda_i}$ is  either 0, or homogeneous and $\deg(\ol{\lambda_i} \ol{g_i})=\deg(\ol{h})$, by~\cite[Exercise 2a, \S 8.3]{CoxLittleOshea15}.   Letting $\lambda_i(x_1,\ldots,x_n)=\bar\lambda_i(1,x_1,\ldots,x_n)$ and dehomogenizing both sides of~\eqref{a:barh} we obtain
$$
h=\sum_{i=1}^n \lambda_i g_i.
$$
 For each $i\in [n]$, we have either $\lambda_i=0$, or $\deg(\lambda_i g_i)\leq  \deg(\bar \lambda_i \bar g_i) =\deg(\bar h)=\deg(h)$. In both cases, we have $\deg(\lambda_i g_i)\leq \deg(h)$.
Hence $\{g_1,\ldots,g_n\}$ forms an H-basis. 

Applying~\Cref{coro:completeintersection} and~\Cref{thm:Hilbertseries}, the Hilbert series of the quotient ring $\bS/ \bI$ is
\begin{align*}
H_{\bS/\bI}(t)=
\frac{ \displaystyle\prod_{i=1}^n \left(  \displaystyle \sum_{j=1}^{\deg(g_i)} t^{j-1} \right)}{1-t}=\left(\displaystyle\prod_{i=1}^n \left(  \displaystyle \sum_{j=1}^{\deg(g_i)} t^{j-1} \right)\right) \left(\sum_{j=0}^{+\infty} t^j\right)= \sum_{d=0}^{+\infty} \left(\sum_{j=0}^d \fc_j \right) t^d,
\end{align*}
where
\begin{equation}\label{a:adeffcd}
\fc_j:=\left |  \left\{ (k_1, \ldots, k_n ) \in \NN^n: k_1+\cdots+k_n=j,\enspace k_i <\deg(g_i),\enspace \forall i\in [n]\right\}\right |.
\end{equation}
The dimension of the vector space $\bS_{\fn}/\bI_{\fn}$ corresponds to the  coefficient of $t^{\fn}$ in the Hilbert series. It is equal to 
$$
\sum_{j=0}^{\fn}\fc_{j}=\prod_{i=1}^n \deg(g_i).
 $$
 We then deduce~\eqref{a:facd}. 
\end{proof}

{
\section{Proof outline for~\Cref{thm:L04}}\label{sec:proofL04}
Let $r=\rank (M_t(\Lambda))$. First, by~\cite{CurtoFialkow96}, there exists $\{v_1,\ldots,v_r\}\subset \R^n$ and  $a_1>0,\ldots,a_r>0$ such that
\begin{equation}\label{eq:Lambda2ts}
\Lambda_{2(t+s)}
=\sum_{i=1}^r a_i [v_i]_{2(t+s)}.
\end{equation}
In addition,  from $\<\Lambda,h_0>=1$ we know that $a_1+\cdots+a_r=1$.
It remains to show that $v_i\in \cF$. For this, we follow the same procedure as in the proof of~\cite[Theorem 1.6]{Laurent04}. We let $q_1,\ldots,q_r$ be the interpolation polynomials at the points $\{v_1,\ldots,v_r\}$ having degree at most $t$. To check that $h_i(v_j)\geq 0$, $\forall i\in [m]$, the proof is exactly the same as in~\cite[Theorem 1.6]{Laurent04}. To check the equality constraints, we note that $\deg(q_jg_i) \leq \deg(g_i)+t\leq 2(t+s)$ following the second condition in~\eqref{eq:scondition}, and therefore
$$
0=\langle\Lambda, q_j g_i \rangle=\langle \Lambda_{2(t+s)}, g_i q_j \rangle=a_j g_i(v_j),\enspace \forall i\in [k], j\in [r].
$$
Here, the first equality follows from $\Lambda \in \cL_{2d}$ and $\deg(q_jg_i)\leq 2d$, the second equality follows from $\deg(q_j g_i)\leq 2(t+s)$. The last equality follows from~\eqref{eq:Lambda2ts} and $\<[v_i]_{2(t+s)}, q_j>=1$ if  $i=j$ and 0 otherwise.
}

\end{appendices}

{
\textbf{ACKNOWLEDGEMENTS} {We would like to express our sincere gratitude to the referees for their valuable feedback and insightful comments, which have significantly contributed to the improvement of this work. Their careful review and constructive suggestions have been instrumental in enhancing the clarity and quality of our manuscript.}	
}

\bibliographystyle{my-plainst}
\bibliography{mybib}   

\end{document}